\documentclass[10pt]{amsart}

\usepackage{amsfonts}
\usepackage{amssymb}
\usepackage{amsmath,amsthm}
\usepackage{amscd}
\usepackage{enumitem}
\usepackage{cancel}
\usepackage{pdfsync}
\usepackage{mathrsfs}
\usepackage{stmaryrd}
\usepackage[colorlinks,linkcolor=blue,citecolor=blue,urlcolor=blue]{hyperref}

\newcommand{\rd}{\mathrm{d}}

\newcommand{\beq}{\begin{equation}}
\newcommand{\eeq}{\end{equation}}
\newcommand{\bea}{\begin{eqnarray}}
\newcommand{\eea}{\end{eqnarray}}
\newcommand{\nn}{\nonumber}
\newcommand\noi{\noindent}

\newcommand{\la}{\lambda}
\newcommand{\bs}{\boldsymbol}
\newcommand{\bk}{\begin{cases}}
\newcommand{\ek}{\end{cases}}
\newcommand{\tbf}{\textbf}

\newtheorem{theorem}{Theorem}
\newtheorem{conjecture}[theorem]{Conjecture}
\newtheorem{definition}[theorem]{Definition}
\newtheorem{proposition}[theorem]{Proposition}
\newtheorem{corollary}[theorem]{Corollary}
\newtheorem{lemma}[theorem]{Lemma}
\theoremstyle{definition}

\newtheorem{remark}[theorem]{\textbf{Remark}}

\newtheorem*{example}{\textbf{Example}}

\begin{document}

\author{Piergiulio Tempesta}
\address{Departamento de F\'{\i}sica Te\'{o}rica II
  (M\'{e}todos Matem\'{a}ticos de la F\'isica), Facultad de Ciencias F\'{\i}sicas,
  Universidad Complutense de Madrid, 28040 -- Madrid, Spain \\ and Instituto de Ciencias
  Matem\'aticas, C/ Nicol\'as Cabrera, No 13--15, 28049 Madrid, Spain}
\email{p.tempesta@fis.ucm.es, piergiulio.tempesta@icmat.es}
\author{Giorgio Tondo}
\address{Dipartimento di Matematica e Geoscienze, Universit\`a  degli Studi di Trieste,
piaz.le Europa 1, I--34127 Trieste, Italy.}
\email{tondo@units.it}

\title[ Higher Haantjes Brackets and Integrability]{Higher Haantjes Brackets and Integrability}

\subjclass[2010]{MSC: 53A45, 58C40, 58A30.}

\date{July 16, 2021}

\begin{abstract}
We propose a new, infinite class of brackets generalizing the Fr\"olicher--Nijenhuis bracket. This class can be reduced to a family of generalized  Nijenhuis torsions recently introduced. In particular, the  Haantjes bracket, the first example of our construction, is relevant in the characterization of Haantjes moduli of operators.

We shall also prove that the vanishing of a higher-level Nijenhuis torsion of a given operator is a sufficient condition for the integrability of its generalized eigen-distributions. This result (which does not require any knowledge of the spectral properties of the operator) generalizes the celebrated Haantjes theorem. The same vanishing condition also guarantees that the operator can be written, in a local chart, in a block-diagonal form.

\end{abstract}

\maketitle

\tableofcontents

\section{Introduction}

In the last two decades, the study of the geometry of Nijenhuis and Haantjes tensors has experienced a resurgence of interest.
The notion of Nijenhuis torsion was introduced  in \cite{Nij1951}, \cite{Nij1955} by A. Nijenhuis  in his study of the integrability of eigen-distributions of operator fields with pointwise distinct eigenvalues. 


In \cite{FN1956}, the graded bracket nowadays called the \textit{Fr\"olicher-Nijenhuis bracket} was defined. This bracket is relevant in several geometric contexts, in particular in the theory of almost-complex structures, as clarified by the
Newlander-Nirenberg theorem \cite{NN1957}, \cite{KMS1993}.
Slightly before, in the seminal paper \cite{Haa1955}, J. Haantjes proposed the fundamental notion of torsion bearing his name. In particular, he proved that
the vanishing of the Haantjes torsion of a $(1,1)$ tensor field is a necessary  condition for the existence of an integrable frame of generalized eigenvectors. This condition is also sufficient in the case of pointwise semisimple operators.

Recently, new and conspicuous applications of Nijenhuis and Haantjes tensors have been found, for instance, in the characterization
of integrable chains of partial differential equations of hydrodynamic type (see e.g. \cite{FeMa}, \cite{BogRey}) and in the study of infinite-dimensional integrable systems (in particular in connection with the celebrated WDVV equations of associativity and the theory of Dubrovin-Frobenius manifolds \cite{MGall13}-\cite{MGall17}). In  \cite{TT2016prepr}--\cite{T2017} we have proposed the notion of Haantjes algebras and the related ones of $\omega \mathcal{H}$ and $P \mathcal{H}$ manifolds as the natural setting for the formulation of the theory of classical finite-dimensional Hamiltonian integrable systems.


The aim of this work is twofold. 

Our first goal is to introduce a new, infinite class of brackets that generalize  the Fr\"olicher--Nijenhuis bracket. The first representative of our class coincides with it; the second one is already a new example, the \textit{Haantjes bracket} $\mathcal{H}_{\bs{A},\bs{B}}(X,Y)$; by means of a recursive procedure we  also define infinitely many novel  higher-level brackets. Here ``bracket'' means that the tensors introduced depend on a \textit{pair} $(\bs{A}, \bs{B})$  of (1,1) tensor fields.

A simple reduction of this family,
obtained when each representative depends on a pair $(\bs{A}, \bs{A})$ of copies of the same operator field, coincides (up to a  constant) with the family of generalized torsions defined independently (and from a different
perspective)  by Y. Kosmann-Schwarzbach in \cite{KS2017} and by ourselves in an early, preprint version of \cite{TT2017} of 2017. 


By means a new ``tower'' of brackets,
we aim to study the geometry of very general families of operators, as the triangularizable ones, that (except in very specific cases) have non-vanishing Haantjes torsion.

We have ascertained  the geometric relevance of our higher brackets in several important situations. Precisely, as stated in Theorem \ref{th:23}, given two commutative semisimple operators, they generate a commutative Haantjes module \cite{TT2017} if and only if their  Haantjes bracket vanishes.
Also, we study several further algebraic properties of this new bracket.

Our second goal is to clarify the geometric meaning of the ``generalized Nijenhuis torsions'' of higher level introduced in \cite{KS2017}, \cite{TT2017}. In Section 5 (Proposition \ref{theo:triangular} and Corollary \ref{cor:ntriangular}) we prove that the vanishing of the generalized Nijenhuis torsion  $\mathcal{\tau}^{(n-1)}_{\boldsymbol{A}}(X,Y)=0$
of level $(n-1)$ of a nilcyclic (i.e. both nilpotent and cyclic) operator field $\bs{A}$ on a manifold of dimension $n$ is necessary for the existence of a local chart where $\bs{A}$ takes a  triangular form (see Eq. \eqref{Anoper}).

The main theorem of the present work, Theorem \ref{th:MainR} of Section 6, concerns the integrability properties of the generalized eigen-distributions (i.e., distributions of generalized eigenvector fields) of an operator field. A seminal result, due to Haantjes \cite{Haa1955}, states that in the case of a \textit{semisimple} operator field, a \textit{necessary and sufficient} condition for the Frobenius integrability of its eigen-distributions of constant rank is that its Haantjes tensor identically vanishes.
However, in the general case of a \textit{non-semisimple} operator, the previous condition is only \textit{sufficient}. Thus, for the infinite class of operators whose Haantjes tensor is not vanishing, no conclusion can be drawn about integrability of their eigen-distributions.

Our main theorem fills this gap. Indeed, we shall prove that the vanishing  of a generalized Nijenhuis torsion $\mathcal{\tau}^{(m)}_{\boldsymbol{A}}(X,Y)$ of level $m$ for some integer $m\geq 1$ provides us with a sufficient condition for the integrability of the generalized eigen-distributions of a given operator field $\bs{A}$. In addition, it ensures the integrability of all of their \textit{direct sums}.  Thus, we are able to construct a \textit{tensorial test} for Frobenius integrability of a very large class of operator fields, which significantly extends the applicability of the original Haantjes torsion criterion.  

The interest of our result, in the spirit of Haantjes's theorem, relies crucially on the fact that, in order to ascertain the integrability properties of a given operator, no knowledge a priori of the spectrum of this operator nor of its eigen-distributions is required.


An important consequence of the main Theorem (see Proposition \ref{cor:Hframe}) is the fact that an operator with a vanishing generalized Nijenhuis torsion (for some $m\geq1$) admits a local coordinate chart where it takes a \textit{block-diagonal} form.

In short, the body of results proposed indicates that all of the infinitely many higher level tensors introduced possess a geometric meaning and are  relevant in applicative contexts.


An important open problem we address in Section 5 is to decompose a generic operator field in a local coordinate chart as the sum of a diagonal operator and another operator, whose generalized Nijenhuis torsion (of a suitable
level) vanishes. 

We also believe that the theory of higher brackets proposed in this work could play a significant role, more generally, in the theory of integrable systems, for instance in the study of generic hydrodynamic-type systems, not possessing
Riemann invariants. For instance, a potentially  interesting area is the study of equations of hydrodynamic type in 2+1 dimensions, namely, equations of the form $u_t=\bs{A}(u)u_x+\bs{B}(u)u_y$, where $\bs{A}(u)$ and $\bs{B}(u)$ are operator fields which not necessarily commute \cite{FeKhu}. It would be interesting, for instance, to classify the pairs of operators $(\bs{A}(u)$, $\bs{B}(u))$ relevant in the theory of hydrodynamic-type systems by means of suitable tensor conditions ensuring integrability (some  
preliminary results are presented in Section 3.3.1.). 


\section{Preliminaries on the Nijenhuis  and Haantjes geometry}
\label{sec:1}
In this section, we shall review some basic notions concerning the geometry of Nijenhuis or Haantjes torsions, following the original papers \cite{Haa1955,Nij1951,FN1956}. Here  we shall focus only on the aspects of the theory which are relevant for the subsequent discussion. 

Let $M$  be a differentiable manifold, $\mathfrak{X}(M)$ the Lie algebra of all smooth vector fields on $M$   and  $\boldsymbol{A}:\mathfrak{X}(M)\rightarrow \mathfrak{X}(M)$ be a smooth $(1,1)$ tensor field (namely, an operator field).  For the sake of simplicity, the expressions ``tensor fields'' and ``operator fields'' will be abbreviated to tensors and operators. In the following, all tensors will be considered to be smooth. 
\begin{definition} \label{def:N}
The
 \textit{Nijenhuis torsion} of $\boldsymbol{A}$ is the vector-valued $2$-form  defined by
\begin{equation} \label{eq:Ntorsion}
\mathcal{T}_ {\boldsymbol{A}} (X,Y):=\boldsymbol{A}^2[X,Y] +[\boldsymbol{A}X,\boldsymbol{A}Y]-\boldsymbol{A}\Big([X,\boldsymbol{A}Y]+[\boldsymbol{A}X,Y]\Big),
\end{equation}
where $X,Y \in \mathfrak{X}(M)$ and $[ \ , \ ]$ denotes the commutator of two vector fields.
\end{definition}
\begin{definition} \label{def:H}
 \noi The \textit{Haantjes torsion} of $\boldsymbol{A}$ is the vector-valued $2$-form defined by
\begin{equation} \label{eq:Haan}
\mathcal{H}_{\boldsymbol{A}}(X,Y):=\boldsymbol{A}^2\mathcal{T}_{\boldsymbol{A}}(X,Y)+\mathcal{T}_{\boldsymbol{A}}(\boldsymbol{A}X,\boldsymbol{A}Y)-\boldsymbol{A}\Big(\mathcal{T}_{\boldsymbol{A}}(X,\boldsymbol{A}Y)+\mathcal{T}_{\boldsymbol{A}}(\boldsymbol{A}X,Y)\Big).
\end{equation}
\end{definition}

\begin{definition}
A Haantjes (Nijenhuis)   operator is a (1,1) tensor  whose  Haantjes (Nijenhuis) torsion identically vanishes.
\end{definition}

A simple, relevant case of Haantjes operator is that of a tensor  $\boldsymbol{A}$  which takes a diagonal form in a local chart $\boldsymbol{x}=(x^1,\ldots,x^n)$:
\begin{equation}
\boldsymbol{A}(\boldsymbol{x})=\sum _{i=1}^n \lambda_{i }(\boldsymbol{x}) \frac{\partial}{\partial x^i}\otimes \rd x^i \ , \label{eq:Ldiagonal}
 \end{equation}
 where $\lambda_{i }(\boldsymbol{x}):=\lambda^{i}_{i}(\boldsymbol{x})$ are the eigenvalues of $\boldsymbol{A}$ and  $\left(\frac{\partial}{\partial x^1},\ldots, \frac{\partial}{\partial x^n}\right) $ are the fields forming the so called \textit{natural frame} associated with the local chart $(x^{1},\ldots,x^{n})$. As is well known, the Haantjes torsion of the diagonal operator \eqref{eq:Ldiagonal} vanishes. 
 
We also recall that two frames  $\{X_1,\ldots,X_n\}$ and $\{Y_1,\ldots,Y_n\}$ are said to be equivalent if $n$ nowhere vanishing smooth
functions $f_i$ exist, such that
\[
 X_i= f_i(\boldsymbol{x}) Y_i \ , \qquad\qquad i=1,\ldots,n \ .
\]
 \begin{definition}\label{def:Iframe} \cite{BCR02}
  An \emph{integrable} frame  is a reference frame equivalent to a natural frame.
\end{definition}

\begin{remark}

We wish to point out that the adjectives  ``diagonalizable'' and ``semisimple'' are both used in the literature, sometimes interchangeably. From now on, we shall call diagonalizable an operator which takes a diagonal form in a natural reference frame (as in formula \eqref{eq:Ldiagonal}), whereas we shall say that  an operator is pointwise semisimple (or semisimple \textit{tout court}) if it admits a local reference frame (not necessarily natural, nor integrable) in which it takes a diagonal form. Diagonalizable operators are obviously semisimple; the converse statement is not true in general. Historically, the problem addressed by Nijenhuis and Haantjes was to ascertain whether a local reference frame constructed out of the eigenvectors of an operator is integrable or not.

\end{remark}

It is interesting to observe that the algebraic properties of Haantjes operators are different, and sometimes richer that those of Nijenhuis operators.
One  useful result is the following (hereafter,  $\boldsymbol{I}: \mathfrak{X}(M)\rightarrow \mathfrak{X}(M)$ will denote the identity operator).
\begin{proposition}  \label{pr:fL} \cite{BogCMP}.
Let  $\boldsymbol{A}$ be a (1,1) tensor. The following identity holds
\begin{equation} \label{eq:LtorsionLocal}
\mathcal{H}_{f \boldsymbol{I}+g \boldsymbol{A}}(X,Y)=g^4\, \mathcal{H}_{ \boldsymbol{A}}(X,Y),
\end{equation}
where $f,g:M \rightarrow \mathbb{R}$ are $C^\infty(M)$ functions.
\end{proposition}
\begin{proof}
See Proposition 1, p. 255 of \cite{BogCMP}.
\end{proof}
\noi Interestingly enough, such a simple property does not hold in the case of a Nijenhuis operator.
\par
\noi Many more examples of Haantjes operators, relevant in classical mechanics and in Riemannian geometry can be found for instance in \cite{RTT2020prepr}--\cite{T2017}.

\section{Haantjes  brackets}
Let $M$ be a differentiable manifold and $\boldsymbol{A},\boldsymbol{B}:\mathfrak{X}(M)\rightarrow \mathfrak{X}(M)$ be two operators.
\subsection{The Fr\"olicher-Nijenhuis bracket}
\begin{definition}
\cite{FN1956} The \textit{Fr\"olicher--Nijenhuis bracket} of $\boldsymbol{A}$ and $\boldsymbol{B}$ is the vector-valued $2$-form 
given by \footnote{
For sake of clarity, in this article we have renounced to the usual unified notation $[ \cdot \ , \cdot ]$ which, depending on the context, should stands for both  the standard Lie bracket of vector fields and  the Fr\"olicher--Nijenhuis bracket
of operators. Instead, we have preferred to maintain the symbol $\llbracket\cdot \ ,\cdot\rrbracket$ for the Fr\"olicher--Nijenhuis bracket  and to introduce  the notation $[ \cdot \ , \cdot ]$  for the Lie bracket of two vector fields and the commutator of two operators.}
\bea \label{eq:binNtors}
\nn
& &\llbracket\bs{A},\bs{B}\rrbracket(X,Y)= 
 \Big(\bs{AB}+\bs{BA}\Big)[X,Y]+[\bs{A}X,\bs{B}Y]+[\bs{B}X,\bs{A}Y] \\
&&-\bs{A}\Big([X,\bs{B}Y]+[\bs{B}X,Y]\Big )
-\bs{B}\Big ([X,\bs{A}Y]+[\bs{A}X,Y]\Big), \qquad X,Y \in \mathfrak{X}(M) \ .
\eea
\end{definition}
\noi The local expression of the components of the Fr\"olicher-Nijenhuis bracket reads
\begin{equation}\label{eq:FNlocal}
\llbracket\bs{A},\bs{B}\rrbracket^i_{jk}=\sum_{l=1}^n \bigg(\bs{A}^l_{[j}\partial_{|l|}\bs{B}^i_{k]}-\bs{A}^i_l\partial_{[j}\bs{B}^l_{k]} +\bs{B}^l_{[j}\partial_{|l|}\bs{A}^i_{k]}-\bs{B}^i_l\partial_{[j}\bs{A}^l_{k]}
\bigg) \ .
\end{equation}

This bracket has relevant geometric applications \cite{NN1957}, in particular in the theory of almost-complex structures and in the detection of obstructions to integrability \cite{KMS1993}.
The  bracket is symmetric and $\mathbb{R}$-linear (but not $C^\infty(M)$-linear)  in  $\boldsymbol{A}$ and $\boldsymbol{B}$.  In fact, it satisfies the identity
\begin{equation} \label{eq:fABFN}
\begin{split}
\llbracket f \bs{A}, g\bs{B}\rrbracket(X,Y)
&=f g \llbracket\bs{A},\bs{B}\rrbracket(X,Y)
-g \big(\bs{\sf{T}}(\bs{B},\bs{A})-\bs{\sf{T}}^T(\bs{B},\bs{A})(\rd f,X,Y)\big) \\ 
&
-f\big(\bs{\sf{T}}(\bs{A},\bs{B})-\bs{\sf{T}}^T(\bs{A},\bs{B})(\rd g,X,Y)\big) \ .
 \end{split}
\end{equation}
Here  
$\bs{\sf{T}}(\bs{A},\bs{B}): \mathfrak{X}^*(M) \times \mathfrak{X}(M)\times \mathfrak{X}(M)\rightarrow \mathfrak{X}(M)$ 
is the  vector-valued $3$--tensor defined by
\begin{equation}\label{eq:T}
\bs{\sf{T}}(\bs{A},\bs{B})(\alpha,X,Y):=(\bs{I}\otimes \bs{AB} -\bs{A}\otimes \bs{B})(\alpha,X,Y) \ .
\end{equation}
We shall denote by $\bs{\sf{T}}^T(\alpha,X,Y):=\bs{\sf{T}}(\alpha,Y,X)$ the transposed   of $\bs{\sf{T}}$ w.r.t. the last two arguments. We also recall that for each operator $\bs{A}$, $\bs{B}$, and for all $\alpha\in \mathfrak{X}^*(M)$, $X,Y \in \mathfrak{X}(M)$,
$$
(\bs{A}\otimes \bs{B}) (\alpha,X,Y)=\langle \alpha, \bs{A}X\rangle \,\bs{B}Y \ .
$$

\noi Note that  for each operator  $ \bs{B}: \mathfrak{X}(M)\rightarrow \mathfrak{X}(M)$, we have
\begin{equation} \label{eq:id}
 \llbracket\bs{I},\bs{B}\rrbracket(X,Y)= \boldsymbol{0}, \qquad \forall X,Y \in \mathfrak{X}(M) \ .
\end{equation}
\noi Choosing $\bs{A}=\bs{B}$ in Eq. \eqref{eq:binNtors}, one gets twice the Nijenhuis torsion:
\[
\llbracket \boldsymbol{A,A}\rrbracket (X,Y) = 2~\mathcal{T}_ {\boldsymbol{A}} (X,Y)  \ .
\]
For all $f,g \in C^\infty(M)$, the following identity holds:
\begin{equation} \label{eq:idtau}
\begin{split}
\tau_{f\bs{A}+g\bs{B}}(X,Y)&=f^2 \tau_{\bs{A}}(X,Y)- f\big(\bs{\sf{T}}(\bs{A},\bs{A})-\bs{\sf{T}}^T(\bs{A},\bs{A})\big)(\rd f,X,Y)
+g^2\tau_{\bs{B}}(X,Y) \\ &- g\big(\bs{\sf{T}}(\bs{B},\bs{B})-\bs{\sf{T}}^T(\bs{B},\bs{B})\big)(\rd g,X,Y) 
+f \,g\, \llbracket \boldsymbol{A},\bs{B}\rrbracket (X,Y) \\&-  f\big(\bs{\sf{T}}(\bs{A},\bs{B})-\bs{\sf{T}}^T(\bs{A},\bs{B})\big)(\rd g,X,Y)
-g\big(\bs{\sf{T}}(\bs{B},\bs{A})-\bs{\sf{T}}^T(\bs{B},\bs{A})\big)(\rd f,X,Y) \ .
\end{split}
\end{equation}
This identity allows us to characterize modules of Nijenhuis operators.
\begin{proposition}
Let  $\bs{A}$ and $\bs{B}$ be two Nijenhuis operators. They generate a module of Nijenhuis operators if and only if the following conditions are fulfilled 
\begin{eqnarray}
&& \llbracket \boldsymbol{A},\bs{B}\rrbracket=\bs{0} \label{eq:FN0} \\
\bs{\sf{T}}(\bs{A},\bs{A})&=&\bs{\sf{T}}^T(\bs{A},\bs{A}) \ ,\qquad\qquad
\bs{\sf{T}}(\bs{B},\bs{B})=\bs{\sf{T}}^T(\bs{B},\bs{B}) \\
\bs{\sf{T}}(\bs{A},\bs{B})&=&\bs{\sf{T}}^T(\bs{A},\bs{B})  \ ,\qquad\qquad
\bs{\sf{T}}(\bs{B},\bs{A})=\bs{\sf{T}}^T(\bs{B},\bs{A}) \ . \label{eq:T0}
\end{eqnarray}

 In particular, they generate a vector space of Nijenhuis operators if and only if their Fr\"olicher--Nijenhuis bracket identically vanishes.
\end{proposition}
\textbf{Examples}. 
\begin{itemize}
\item[i)]
Let us consider the couple of Nijenhuis operators $\bs{A}= f(\bs{x}) \frac{\partial}{\partial x^1}\otimes \rd x^n$,  $\bs{B}= g(\bs{x}) \frac{\partial}{\partial x^1}\otimes \rd x^n$. They satisfy conditions \eqref{eq:FN0}-\eqref{eq:T0}; then, they generate a module of Nijenhuis operators.
\item[ii)]
Let us consider the couple of Nijenhuis operators $\bs{A}= f(x^i) \frac{\partial}{\partial x^i}\otimes \rd x^i$,  $\bs{B}= g(x^k) \frac{\partial}{\partial x^k}\otimes \rd x^k$, whose Fr\"olicher-Nijenhuis bracket vanishes. In fact, they generate a vector space of Nijenhuis operators. 
\end{itemize}
\vspace{2mm}
Inspired by the previous construction, we shall introduce a novel ``tower'' of higher-level brackets. The first step is to generalize the Fr\"olicher--Nijenhuis bracket.
\subsection{A new family of higher brackets}
Hereafter, we shall present the main algebraic construction of this work, namely the recursive definition of an infinite class of new brackets of couples of  operators.
\begin{definition} \label{GHBT}
Let $M$ be a differentiable manifold of dimension $n$ and let $\boldsymbol{A},\boldsymbol{B}:\mathfrak{X}(M)\rightarrow \mathfrak{X}(M)$ be two $(1,1)$ tensors.
 The
 \textit{Haantjes bracket of level $m\in\mathbb{N}\backslash\{0\}$} of $\boldsymbol{A}$ and $\boldsymbol{B}$  is the vector--valued $2$--form defined, for any  $X,Y \in \mathfrak{X}(M)$, by the relations
\bea \label{GHTn}
\nn & & \!\!\!\!\!\!\!\!\mathcal{H}^{(1)}_{\boldsymbol{A,B}} (X,Y) := \llbracket\bs{A}, \bs{B}\rrbracket(X,Y) \\
\nn \text{and} \\
\nn & &\!\!\!\!\!\!\!\!  \mathcal{H}^{(m)}_{\boldsymbol{A,B}}(X,Y):=\Big(\bs{AB}+ \bs{BA}\Big)\mathcal{H}_{\boldsymbol{A,B}}^{(m-1)}(X,Y) +\mathcal{H}_{\boldsymbol{A,B}}^{(m-1)}(\bs{A}X,\bs{B}Y)+\mathcal{H}_{\boldsymbol{A,B}}^{(m-1)}(\bs{B}X,\bs{A}Y)
 \\
\nn &&\!\!\!\!\!\!\!\! -\bs{A}\Big(\mathcal{H}_{\boldsymbol{A,B}}^{(m-1)}(X,\bs{B}Y) +\mathcal{H}_{\boldsymbol{A,B}}^{(m-1)}(\bs{B}X,Y)\big)-\bs{B}\big(\mathcal{H}_{\boldsymbol{A,B}}^{(m-1)}(X,\bs{A}Y)+
\mathcal{H}^{(m-1)}_{\boldsymbol{A,B}}(\bs{A}X,Y)\Big), \quad m\geq 2 \ .
\eea
\beq
\eeq
\end{definition}
None of these brackets for $m\geq 2$ is $\mathbb{R}$-linear in $\bs{A}$ and $\bs{B}$; however, they are symmetric in the interchange of $\bs{A}$ and $\bs{B}$.

\noi The following statement can be useful for computational purposes.
\begin{lemma}
The expression in local coordinates of the   Haantjes brackets of level $m$, for $m\geq 2$,  reads
\begin{eqnarray}\label{eq:GenHaanLocal}
\nn (\mathcal{H}^{(m)}_{\boldsymbol{A,B}})^i_{jk}&=&   \sum_{\alpha,\beta=1}^n\bigg(
\boldsymbol{A}^i_\alpha \boldsymbol{B}^\alpha_\beta(\mathcal{H}^{(m-1)}_{\boldsymbol{A,B}})^\beta_{jk}  +
\boldsymbol{B}^i_\alpha \boldsymbol{A}^\alpha_\beta(\mathcal{H}^{(m-1)}_{\boldsymbol{A,B}})^\beta_{jk}+
(\mathcal{H}^{(m-1)}_{\boldsymbol{A,B}})^i_{\alpha \beta}\boldsymbol{A}^\alpha_j \boldsymbol{B}^\beta_k \\ \nn &+&
(\mathcal{H}^{(m-1)}_{\boldsymbol{A,B}})^i_{\alpha \beta}\boldsymbol{B}^\alpha_j \boldsymbol{A}^\beta_k -
\boldsymbol{A}^i_\alpha\Big( (\mathcal{H}^{(m-1)}_{\boldsymbol{A,B}})^\alpha_{j \beta} \boldsymbol{B}^\beta_k +
 (\mathcal{H}^{(m-1)}_{\boldsymbol{A,B}})^\alpha_{\beta k} \boldsymbol{B}^\beta_j \Big) \\  &-&\boldsymbol{B}^i_\alpha\Big( (\mathcal{H}^{m-1}_{\boldsymbol{A,B}})^\alpha_{j \beta} \boldsymbol{A}^\beta_k+
 (\mathcal{H}^{(m-1)}_{\boldsymbol{A,B}})^\alpha_{\beta k} \boldsymbol{A}^\beta_j \Big)
 \bigg) \ .
\end{eqnarray}
\end{lemma}
\begin{proof}
This formula comes directly from the expression in local coordinates of the Fr\"olicher-Nijenhuis bracket \eqref{eq:FNlocal}, applied to the recursive formula \eqref{GHTn}.
\end{proof}
If we take $\bs{A}=\bs{B}$, then the previous family of brackets reduces to the generalized torsions proposed independently in \cite{KS2017} and \cite{TT2017}. Here we remind
the main definition of that construction, since it will be crucial in the subsequent discussion.
\begin{definition} \label{df:mtorsion}
Let $\boldsymbol{A}:\mathfrak{X}(M)\to \mathfrak{X}(M)$ be a (1,1) tensor. The generalized Nijenhuis torsion of $\boldsymbol{A}$ of level $m$, for each integer $m\geq 1$, is the vector--valued $2$--form defined by
\bea \label{GNTn}
\nn \mathcal{\tau}^{(m)}_{\boldsymbol{A}}(X,Y):&=&\frac{1}{2^{m}}\mathcal{H}^{(m)}_{\boldsymbol{A,A}} (X,Y)=\boldsymbol{A}^2\mathcal{\tau}^{(m-1)}_{\boldsymbol{A}}(X,Y)+
\mathcal{\tau}^{(m-1)}_{\boldsymbol{A}}(\boldsymbol{A}X,\boldsymbol{A}Y) \\ &-& 
\boldsymbol{A}\Big(\mathcal{\tau}^{(m-1)}_{\boldsymbol{A}}(X,\boldsymbol{A}Y)+\mathcal{\tau}^{(m-1)}_{\boldsymbol{A}}(\boldsymbol{A}X,Y)\Big), \quad X,Y \in \mathfrak{X}(M) \ .
\eea
Here the notation $\tau_{\boldsymbol{A}}^{(0)}(X,Y):= [X,Y]$,  $\tau_{\boldsymbol{A}}^{(1)}(X,Y):=\tau_{\boldsymbol{A}}(X,Y)$
and $\tau_{\boldsymbol{A}}^{(2)}(X,Y):=\mathcal{H}_{\boldsymbol{A}}(X,Y)$ is used.
\end{definition}
\noi We also remind a useful formula, proved in \cite{KS2017} (Section 4.6), by means of a suitable polynomial representation of $(1,2)$ tensors:
\begin{equation} \label{eq:K17}
\mathcal{\tau}^{(m)}_{{\boldsymbol{A}}}(X,Y)=\sum_{p=0}^m \sum_{q=0}^{m} (-1)^{2m-p-q} \binom{m}{p} \binom{m}{q} \bs{A}^{p+q}\big[ \bs{A}^{m-p}X,\bs{A}^{m-q} Y\big ] \ .
\end{equation}
Alternatively, this formula can also be proved by induction over $m$.

Hereafter, we shall discuss some  relevant properties of the new brackets \eqref{GHTn}.
\begin{lemma}  Let $M$ be a differentiable manifold and $\boldsymbol{A},\boldsymbol{B}:\mathfrak{X}(M)\rightarrow \mathfrak{X}(M)$ be two operators. For any $X,Y \in \mathfrak{X}(M)$, we have
\begin{equation} \label{eq:HaanIBn}
\mathcal{H}^{(m)}_{\boldsymbol{\bs{I},B}} (X,Y)=\bs{0}  , \qquad \quad  \ m\in\mathbb{N}\backslash \{0\}\ .
\end{equation}
Moreover, if $[\boldsymbol{A}, \boldsymbol{B}]=0$, we have
\beq \label{eq:HfnAB}
\mathcal{H}^{(m)}_{f \bs{A,}\, g \bs{B}} (X,Y)= f^m\, g^m\, \mathcal{H}^{(m)}_{\boldsymbol{A,B}} (X,Y), \qquad\qquad m\in\mathbb{N}\backslash \{0,1\} \ .
\eeq
Consequently,
\begin{equation} \label{eq:HaanfIBn}
\mathcal{H}^{(m)}_{f \boldsymbol{\bs{I},B}} (X,Y)=\bs{0}  \ .
\end{equation}
\end{lemma}
\begin{proof}
Relation \eqref{eq:HaanIBn}  is obtained by induction over $m$, starting with the case $m=1$ already proved in Eq. \eqref{eq:id}. 
Similarly, property \eqref{eq:HfnAB} can be proved by induction over $m\geq 2$; the case $m=2$ simply requires a direct calculation. Equation  \eqref{eq:HaanfIBn} is an immediate consequence of Eqs. \eqref{eq:HaanIBn} and \eqref{eq:HfnAB}.
\end{proof}
\begin{proposition} Let $M$ be a differentiable manifold and $\boldsymbol{A},\boldsymbol{B}:\mathfrak{X}(M)\rightarrow \mathfrak{X}(M)$  two commuting operators. For any $f,g,h,k \in C^\infty(M)$ , $X,Y \in \mathfrak{X}(M)$ and for each integer $m\geq 2$, we have
\begin{equation} \label{eq:34}
\mathcal{H}^{(m)}_{f \boldsymbol{I}+g\boldsymbol{A},h\boldsymbol{I}+k\boldsymbol{B}} (X,Y)=g^mk^m\mathcal{H}^{(m)}_{\boldsymbol{A},\boldsymbol{B}} (X,Y)  \ .
\end{equation}
\end{proposition}
\begin{proof}
The formula can be proved by  induction over $m$, starting with the case $h=0$ and $k=1$.
Then, the result follows as a consequence of  the symmetry w.r.t. the interchange of the first and  second operator.
\end{proof}
\begin{corollary} \label{corollary:3}
Let $\bs{A}:\mathfrak{X}(M)\to \mathfrak{X}(M)$ be an operator. Then, for all $f\in C^{\infty}(M)$ the relations 
\begin{eqnarray} \label{tauind}
\mathcal{\tau}^{(m)}_{\bs{I}}(X,Y) &=&\bs{0}\ , \quad\qquad \qquad \qquad m\in\mathbb{N}\backslash \{0\} \\\label{tauind2}
\mathcal{\tau}^{(m)}_{f \bs{I}+ g \bs{A}}(X,Y)&=& g^{2m}\mathcal{\tau}^{(m)}_{\boldsymbol{A}}(X,Y), \qquad m\in\mathbb{N}\backslash \{0,1\} \ .
\end{eqnarray}
hold.
\end{corollary}
\begin{proof}
Eqs. \eqref{tauind} and \eqref{tauind2} are obtained by imposing  $g=1$, $h=f$, $k=1$, $\bs{B}=\bs{A}$ into Eq. \eqref{eq:34}.
\end{proof}

\noi The value $m=1$  has been excluded in Eq. \eqref{tauind2}, since for this case a separate formula  for the Nijenhuis torsion holds:
\begin{equation} \label{eq:fNtorsion}
\mathcal{T}_{g \boldsymbol{A}}(X,Y)=g^2\mathcal{T}_{\boldsymbol{A}}(X,Y)
-g \big(\bs{\sf{T}}(\bs{A},\bs{A})-\bs{\sf{T}}^T(\bs{A},\bs{A})\big)(\rd g,X,Y) \ .
\end{equation}
This equation can be easily obtained from Eq. \eqref{eq:fABFN}, choosing $f= g$ and $\bs{B}=  \bs{A}$.

Let us consider in more detail the properties of  the Haantjes bracket of level $m=2$ of  two arbitrary commuting operators. 
Hereafter the notation $\mathcal{H}_ {\boldsymbol{A,B}}(X,Y):= \mathcal{H}^{(2)}_{\boldsymbol{A,B}} (X,Y)$
will be used.
\begin{proposition} \label{prop11}
Let $M$ be a differentiable manifold, $f, g\in C^\infty(M)$ and let $\boldsymbol{A},\boldsymbol{B}:\mathfrak{X}(M)\rightarrow \mathfrak{X}(M)$ be two $(1,1)$ tensors.  Then, the following identity holds:
\begin{multline} \label{eq:prop11}
\mathcal{H}_{f \bs{A,}\,g\bs{B}} (X,Y)
= f^2 g^2\, \mathcal{H}_{\boldsymbol{A,B}} (X,Y)\\
+
f g^2\Bigg (\bs{I}\otimes  [\bs{A},\bs{B}[\boldsymbol{A,B}]]-\bs{B}\otimes  [\bs{A},[\boldsymbol{A,B}]]-
\bs{A}\otimes \bs{B} [\boldsymbol{A,B}]+\bs{BA}\otimes [\boldsymbol{A,B}]\Bigg)( \rd f, X,Y) \\
-
f g^2\Bigg (\bs{I}\otimes  [\bs{A},\bs{B}[\boldsymbol{A,B}]]-\bs{B}\otimes  [\bs{A},[\boldsymbol{A,B}]]-
\bs{A}\otimes \bs{B} [\boldsymbol{A,B}]+\bs{BA}\otimes [\boldsymbol{A,B}]\Bigg)( \rd f, Y,X) \\
+
f^2 g\Bigg (\bs{I}\otimes  [\bs{B},\bs{A}[\boldsymbol{B,A}]]-\bs{A}\otimes  [\bs{B},[\boldsymbol{B,A}]]-
\bs{B}\otimes \bs{A} [\boldsymbol{B,A}]+\bs{AB}\otimes [\boldsymbol{B,A}]\Bigg)( \rd g, X,Y) \\
-
f^2 g\Bigg (\bs{I}\otimes  [\bs{B},\bs{A}[\boldsymbol{B,A}]]-\bs{A}\otimes  [\bs{B},[\boldsymbol{B,A}]]-
\bs{B}\otimes \bs{A} [\boldsymbol{B,A}]+\bs{AB}\otimes [\boldsymbol{B,A}]\Bigg)( \rd g, Y,X) \ .
\end{multline}
\end{proposition}
Formula \eqref{eq:prop11} can be derived by a direct (although cumbersome) calculation.

From Definition \ref{GHBT},  by means of some algebraic manipulations one can derive another useful result.
\begin{lemma}\label{lm:HaanApfI}
Let $M$ be a differentiable manifold, $f\in C^\infty(M)$ and let $\boldsymbol{A},\boldsymbol{B}:\mathfrak{X}(M)\rightarrow \mathfrak{X}(M)$ be two $(1,1)$ tensors.  Then,  for all $f\in C^\infty(M)$ we have
\begin{eqnarray*}
\mathcal{H}_{f \boldsymbol{I}+\boldsymbol{A},\boldsymbol{B}} (X,Y)
=\mathcal{H}_{\boldsymbol{A},\boldsymbol{B}} (X,Y)&+&(\bs{I}\otimes \bs{B}+\bs{B}\otimes \bs{I})[\bs{A},\bs{B}](\rd f,X,Y)\\
&-&
(\bs{I}\otimes \bs{B}+\bs{B}\otimes \bs{I})[\bs{A},\bs{B}](\rd f,Y,X) \ .
\end{eqnarray*}

\end{lemma}
%
\noi The following result clarifies the geometric meaning of the Haantjes  bracket of level 2. 
\begin{lemma} \label{theo:HBT}
 Let $M$ be a  differentiable manifold  and $\bs{A}$, $\bs{B}:\mathfrak{X}(M)\rightarrow \mathfrak{X}(M)$  two $(1,1)$ tensors which can be simultaneously diagonalized in a local chart of $M$.
Then for any $X,Y \in \mathfrak{X}(M)$, the  Haantjes  bracket $\mathcal{H}_{\boldsymbol{A,B}} (X,Y)$ vanishes.
\end{lemma}
\begin{proof}
We denote by $\bs{A}^i_i$ and $\bs{B}^j_j$ the non-vanishing components of $\bs{A}$ and $\bs{B}$ respectively. Then, in a local chart where the operators  diagonalize simultaneously, using Eq. \eqref{eq:GenHaanLocal}  we get, by means of a direct calculation,
$$
(\mathcal{H}_{\boldsymbol{A,B}})^i_{j k}=\llbracket\bs{A},\bs{B}\rrbracket^i_{jk} \bigg( 2 \bs{A}^i_i\bs{B}^i_i+\bs{A}^j_j\bs{B}^k_k+\bs{A}^k_k \bs{B}^j_j-\bs{A}^i_i (\bs{B}^k_k+\bs{B}^j_j ) -\bs{B}^i_i (\bs{A}^k_k +\bs{A}^j_j) \bigg)
$$
where $\llbracket\bs{A},\bs{B}\rrbracket^i_{jk}$ is explicitly given in formula \eqref{eq:FNlocal}. If  $i$, $j$, $k$ are all different, then $ \llbracket \bs{A},\bs{B} \rrbracket^i_{jk}=0$. Moreover, if $i=k \neq j$ or $i=j\neq k$, the sum in the r.h.s. vanishes as well.
\end{proof}
The latter property, which does not hold in the case of the  Fr\"olicher-Nijenhuis bracket, is analogous to the one valid for the standard Haantjes torsion of diagonalizable operators. In fact, the Haantjes torsion vanishes, whereas the Nijenhuis one does not necessarily.

\subsection{ Haantjes brackets and Haantjes modules}

In the following analysis, we shall illustrate the algebraic meaning of Haantjes brackets of level 2. As we will show, they play a crucial role in the study of the $C^{\infty}(M)$-modules of Haantjes operators, that we shall call \textit{Haantjes modules}.

\begin{definition}\label{def:HM}
A Haantjes module is a pair    $(M, \mathscr{H}_{\mathcal{M}})$ which satisfies the following  conditions:
\begin{itemize}
\item
$M$ is a differentiable manifold of dimension $\mathrm{n}$;
\item
$ \mathscr{H}_{\mathcal{M}}$ is a set of Haantjes  operators $\boldsymbol{K}:\mathfrak{X}(M)\rightarrow \mathfrak{X}(M)$   such that
\begin{equation}\label{eq:Hmod}
\mathcal{H}_{\big( f\boldsymbol{K_1} +
                             g\boldsymbol{K}_2\big)}(X,Y)= \boldsymbol{0}
 \ , \qquad\forall X, Y \in \mathfrak{X}(M) \ , \quad \forall f,g \in C^\infty(M)\  ,\quad \forall \boldsymbol{K}_1,\boldsymbol{K}_2 \in  \mathscr{H}_{\mathcal{M}}.
\end{equation}
\end{itemize}
\end{definition}
Thus, a Haantjes module is a free module of Haantjes operators  over the ring of smooth functions on $M$. If property \eqref{eq:Hmod} is satisfied only when $f,g$ are real constants, we shall use the denomination of \textit{Haantjes vector space}.

\vspace{3mm}

We determine now the tensorial compatibility conditions ensuring the existence of the Haantjes module generated by two  arbitrary  Haantjes operators $\bs{A}$, $\bs{B}:\mathfrak{X}(M)\to \mathfrak{X}(M)$. First, we construct these conditions in full generality, namely for \textit{non-semisimple}, non-commuting Haantjes operators.  Then, we shall restrict to the important case of semisimple, commuting operators, which arises for instance in Hamiltonian classical mechanics, in the discussion of separable systems \cite{RTT2020prepr}, \cite{TT2016prepr}. 

\subsubsection{The general case}
We shall start our analysis with 
the following identity, valid for all $f,g$ $\in$ $C^\infty(M)$, $X,Y$$\in$ $\mathfrak{X}(M)$:
\begin{equation} \label{eq:HaanA+B}
\mathcal{H}_{f\bs{A}+g\bs{B}}(X,Y)=f^4\mathcal{H}_{\boldsymbol{A}}(X,Y)+g^4\mathcal{H}_{\boldsymbol{B}}(X,Y)+
\sigma_{f\bs{A},g\bs{B}}(X,Y) \ ,
\end{equation}
where  
\begin{equation}
 \sigma_{\boldsymbol{A,B}} (X,Y):=\mathcal{H}_{\boldsymbol{A,B}} (X,Y)+\mathcal{H}_1(\boldsymbol{A,B}) (X,Y) 
+ \mathcal{H}_2(\boldsymbol{A,B}) (X,Y)+\mathcal{H}_2 (\boldsymbol{B,A}) (X,Y)
\end{equation}
   and  $\mathcal{H}_1 (\boldsymbol{A,B})$,  $\mathcal{H}_2 (\boldsymbol{A,B})$ are the  vector valued $2$-forms
\begin{equation}\label{eq:H12}
\begin{split}
&\mathcal{H}_1(\boldsymbol{A,B})(X,Y):=\bs{B}^2\tau_{\bs{A}}(X,Y)+\tau_{\bs{A}}(\bs{B}X,\bs{B}Y) -\bs{B}\big( \tau_{\bs{A}}(\bs{B}X,Y)+\tau_{\bs{A}}(X,\bs{B}Y)\big) \\
&\qquad\qquad+\bs{A}^2\tau_{\bs{B}}(X,Y)+\tau_{\bs{B}}(\bs{A}X,\bs{A}Y) -\bs{A}\big( \tau_{\bs{B}}(\bs{A}X,Y)+\tau_{\bs{B}}(X,\bs{A}Y)\big )\ ,\\
&  \mathcal{H}_2(\boldsymbol{A,B}) (X,Y):
=(\bs{A}\bs{B}+\bs{B}\bs{A})\tau_{\bs{A}}(X,Y)+
\tau_{\bs{A}}(\bs{A}X,\bs{BY}) +\tau_{\bs{A}}(\bs{B}X,\bs{A}Y) \\ 
&\qquad\qquad-\bs{A}\big(  \tau_{\bs{A}}(\bs{B}X,Y)+\tau_{\bs{A}}(X,\bs{B}Y)\big)
-\bs{B}\big(  \tau_{\bs{A}}(\bs{A}X,Y)+\tau_{\bs{A}}(X,\bs{A}Y\big)\big)\\
&+ \bs{A}^2 \llbracket\bs{A},\bs{B}\rrbracket(X,Y)+
\llbracket\bs{A},\bs{B}\rrbracket(\bs{A}X,\bs{A}Y)
-\bs{A}\big( \llbracket\bs{A},\bs{B}\rrbracket(\bs{A}X,Y)+\llbracket\bs{A},\bs{B}\rrbracket(X,\bs{A}Y)\big) 
\ .
 \end{split}
 \end{equation}
\noi They represent new, auxiliary brackets which complement the role of the  Haantjes bracket $\mathcal{H}_{\bs{A},\bs{B}}$. 

\noi Now, some technical results are in order.
Let us introduce the three vector-valued $3$--tensors 
\[
\bs{\sf{T}}_i: \mathfrak{X}^*(M) \times \mathfrak{X}(M)\times \mathfrak{X}(M)\rightarrow \mathfrak{X}(M), (\alpha,X,Y)\mapsto \bs{\sf{T}}_i(\alpha,X,Y),
\quad i=1,2,3:
\]
\begin{equation}
\bs{\sf{T}}_1(\bs{A},\bs{B})(\alpha,X,Y):=\bigg(  \Big( \bs{\sf{T}}(\bs{A},\bs{B})+\bs{\sf{T}}'(\bs{A},\bs{B})\Big)[ \bs{A},\bs{B} ]\bigg)(\alpha  , X, Y)  \ ,
 \end{equation}
\begin{equation}
\begin{split}
 \bs{\sf{T}}_2(\bs{A},\bs{B})(\alpha,X,Y):=&  \bigg(  \Big( \bs{\sf{T}}(\bs{A},\bs{A})+\bs{\sf{T}}'(\bs{A},\bs{A})\Big)[ \bs{A},\bs{B} ]\bigg)(\alpha  , X, Y)  \ ,  \\
\bs{\sf{T}}_3(\bs{A},\bs{B})(\alpha,X,Y):=& 
\bigg(  \Big( \bs{\sf{T}}(\bs{B},\bs{B})+\bs{\sf{T}}'(\bs{B},\bs{B})\Big)[ \bs{A},\bs{B} ]\bigg)(\alpha  , X, Y)  \ ,
 \end{split}
\end{equation}
where $\bs{\sf{T}}(\bs{A},\bs{B})$ is defined in Eq. \eqref{eq:T} and
\begin{equation}
\bs{\sf{T}}'(\bs{A},\bs{B})(\alpha  , X, Y) :=( \bs{AB} \otimes \bs{I}-\bs{A}\otimes \bs{B})(\alpha,X,Y) \ .
\end{equation}
These brackets satisfy the relations
\begin{equation}\label{eq:Tprop}
\bs{\sf{T}}_1(\bs{A},\bs{B})=- \bs{\sf{T}}_1(\bs{B},\bs{A}) \ ,
\qquad \bs{\sf{T}}_3(\bs{A},\bs{B})=- \bs{\sf{T}}_2(\bs{B},\bs{A}) \ .
\end{equation}
\begin{lemma} \label{Lemma:19}
 Let  $\boldsymbol{A},\boldsymbol{B}:\mathfrak{X}(M)\rightarrow \mathfrak{X}(M)$ be two   operators. 
 For all $f,g\in C^\infty(M)$ and $X,Y \in \mathfrak{X}(M)$, the following identities hold:
\begin{equation} \label{eq:27}
\begin{split}
&\mathcal{H}_{f\boldsymbol{A},g\bs{B}}(X,Y)+\mathcal{H}_1(f\boldsymbol{A},g\bs{B})(X,Y)= 
f^2 g^2\Big(\mathcal{H}_{\boldsymbol{A},\bs{B}} (X,Y)+ \mathcal{H}_1(\boldsymbol{A,B})(X,Y) \Big)\\
&+f g^2\Big(\bs{T}_1(\bs{A},\bs{B})-\bs{T}_1^T(\bs{A},\bs{B})\Big)(\rd f,X,Y)-f^2 g\Big(\bs{T}_1(\bs{A},\bs{B})-\bs{T}_1^T(\bs{A},\bs{B})\Big)(\rd g,X,Y)\\
&\mathcal{H}_2(f\boldsymbol{A},g\bs{B})(X,Y)=
 f^3 g\,\mathcal{H}_2(\boldsymbol{A,B})(X,Y) \big)
 +f^2g\, 
\Big(\bs{T}_2(\bs{A},\bs{B})-\bs{T}_2^T(\bs{A},\bs{B})\Big)(\rd f,Y,X)\big)\\
&\qquad\qquad\qquad\qquad-f^3\Big(\bs{T}_2(\bs{A},\bs{B})-\bs{T}_2^T(\bs{A},\bs{B})\Big)(\rd g,X,Y) \\
&\mathcal{H}_2(g\bs{B},f\boldsymbol{A})(X,Y)=
 f g^3\,\mathcal{H}_2(\bs{B},\boldsymbol{A})(X,Y) \big)
 +g^3\, 
\Big(\bs{T}_3(\bs{A},\bs{B})-\bs{T}_3^T(\bs{A},\bs{B})\Big)(\rd f,Y,X)\big)\\
&\qquad\qquad\qquad\qquad-fg^2\Big(\bs{T}_3(\bs{A},\bs{B})-\bs{T}_3^T(\bs{A},\bs{B})\Big)(\rd g,X,Y).
 \end{split}
\end{equation}
\end{lemma}

\noi From Eqs. \eqref{eq:27} we get the identity  
\begin{equation} \label{eq:sigmaAB}
\begin{split}
\sigma_{f\bs{A},g\bs{B}}(X,Y)&=f^2g^2 \Big(  \mathcal{H}(\boldsymbol{A,B})(X,Y)+\mathcal{H}_1(\boldsymbol{A,B})(X,Y)\Big ) \\
&+f^3 g \, \mathcal{H}_2(\boldsymbol{A,B})(X,Y) + fg^3 \, \mathcal{H}_2(\boldsymbol{B,A})(X,Y) \\
&+\left(\Big(g^3 (\bs{\sf{T}}_3-\bs{\sf{T}}_3^T)+g^2 f (\bs{\sf{T}}_1-\bs{\sf{T}}_1^T)+g f^2 (\bs{\sf{T}}_2-\bs{\sf{T}}_2^T)\Big)(\boldsymbol{A,B})\right) (\rd f, X,Y)  \\
&+\left(\Big(f^3 (\bs{\sf{T}}_3-\bs{\sf{T}}_3^T)+f^2g \,(\bs{\sf{T}}_1-\bs{\sf{T}}_1^T)+fg^2 (\bs{\sf{T}}_2-\bs{\sf{T}}_2^T)\Big)(\boldsymbol{B,A})\right) (\rd g, X,Y) \ . 
\end{split}
 \end{equation}
 From Eqs. \eqref{eq:HaanA+B}, \eqref{eq:sigmaAB} and \eqref{eq:Tprop},   taking into account the previous discussion, we obtain our tensorial characterization of Haantjes modules. 
\begin{theorem}
Let $M$ be a differentiable manifold and let $\boldsymbol{A},\boldsymbol{B}:\mathfrak{X}(M)\rightarrow \mathfrak{X}(M)$ be two Haantjes operators. Then, $\bs{A}$ and $\bs{B}$ generate a Haantjes module if and only if the differential  conditions 
\beq \label{eq:DC}
  \mathcal{H}_{\boldsymbol{A,B}} +\mathcal{H}_1(\boldsymbol{A,B})= 0 \ , \qquad
  \mathcal{H}_2(\boldsymbol{A,B}) = 0 \ , \qquad \mathcal{H}_2(\boldsymbol{B,A})  =0 \ ,
\eeq
together with the algebraic conditions 
\beq\label{eq:AC}
 \bs{\sf{T}}_1(\boldsymbol{A,B}) =\bs{\sf{T}}_1^T(\boldsymbol{A,B}) \ , \qquad 
 \bs{\sf{T}}_2(\boldsymbol{A,B}) =\bs{\sf{T}}_2^T(\boldsymbol{A,B}) \ ,
\qquad \bs{\sf{T}}_2(\boldsymbol{B,A}) = \bs{\sf{T}}_2^T(\boldsymbol{B,A})
 \eeq
are satisfied. In particular, $\bs{A}$ and $\bs{B}$ generate a Haantjes vector space if and only if the differential conditions \eqref{eq:DC} are fulfilled.
\end{theorem}

\begin{corollary} \label{Ch:21}
Let $\bs{A}$ and $\bs{B}$ two commuting Haantjes operators. 
They generate a  Haantjes module if and only if conditions \eqref{eq:DC} are fulfilled.
\end{corollary}
Some examples of applications are in order.
\begin{itemize}
\item
\textit{Haantjes moduli}: In \cite{TT2016prepr}, a  Haantjes module of operators for the Post-Winternitz  superintegrable system has been constructed. The two generators of the module, which do not commute, fulfill conditions \eqref{eq:DC} and \eqref{eq:AC}.
\item
\textit{Haantjes vector spaces}: In  \cite{FeKhu}, (2+1)-dimensional hydrodynamic type systems of the form $u_t=\bs{A}(u)u_x+\bs{B}(u)u_y$, where $u=u(x,y,t)$ have been considered. In the case of the generalized Benney system  and of an isoentropic gas, the two associated operators $\bs{A}(u)$ and $\bs{B}(u)$   do not commute. Also, they fulfil  Eq. \eqref{eq:DC} but not \eqref{eq:AC}.  Therefore, these operators generate a Haantjes vector space.
\end{itemize}

\subsubsection{The semisimple, Abelian case} In the  previous analysis, the Haantjes operators $\bs{A}$ and $\bs{B}$ are not supposed to be semisimple. Let us show that if $\bs{A}$ and $\bs{B}$ commute and they are semisimple, then the three differential conditions \eqref{eq:DC} reduce to the vanishing of the Haantjes bracket $\mathcal{H}_{\boldsymbol{A,B}}$. 
Although this is a special case of the previous construction, it requires an \textit{ad hoc} analysis.

To this aim,   we need to evaluate  the Frolicher-Nijenhuis bracket, as well as the brackets  $\mathcal{H}_{\boldsymbol{A,B}}$, $\mathcal{H}_1(\boldsymbol{A,B})$ and  $\mathcal{H}_2(\boldsymbol{A,B})$ over two common eigenvectors $X_{\mu }$ and $Y_{\nu }$ of two (arbitrary) operators $\bs{A}$ and $\bs{B}$ (the details of the calculation are reported in Appendix 6.3).

\begin{proposition} \label{pr:H1230}
Let $\bs{A}$ and $\bs{B}$ two  Haantjes operators and  $X_\mu$, $Y_\nu$ two common eigenvectors. Then
\begin{equation}\label{eq:H120}
\mathcal{H}_1(\boldsymbol{A,B})(X_{\mu},Y_{\nu})=\bs{0} \ .
\end{equation}
In addition, if $\bs{A}$ and $\bs{B}$ also commute, then 
\begin{equation}\label{eq:H20}
\mathcal{H}_2(\boldsymbol{A,B})(X_{\mu},Y_{\nu})=\bs{0}\ , \qquad \mathcal{H}_2(\boldsymbol{B,A})(X_{\mu},Y_{\nu})=\bs{0} \ .
\end{equation}
\end{proposition}
\begin{proof}
From Eq. \eqref{eq:ABauto} in Appendix 6.3 and the assumption that $\bs{A}$ and $\bs{B}$ are Haantjes operators it follows that 
\begin{equation}
[X_\mu, Y_\nu ] \in \big (\ker(\bs{A}-\mu_1\bs{I}) \oplus  \ker(\bs{A}-\nu_1\bs{I}) \big) \cap
 \big (\ker  (\bs{B}-\mu_2\bs{I}) \oplus \ker (\bs{B}-\nu_2\bs{I})   \big) \ .
\end{equation}

Consequently, it is  evident from Eqs. \eqref{eq:H12XY}   
that Eq. \eqref{eq:H120} holds for any pair of Haantjes operators $\bs{A}$ and $\bs{B}$. In addition, if $\bs{A}$ and $\bs{B}$ \emph{commute}, from Eqs. \eqref{eq:H12XY}  it follows that Eqs. \eqref{eq:H20} also hold.
\end{proof}
 
We can now formulate our main result concerning the characterization of Haantjes modules.
\begin{theorem}\label{th:23}
Let $\bs{A},\bs{B}: \mathfrak{X}(M)\to \mathfrak{X}(M)$ two  commuting semisimple Haantjes operators. 
They generate a  Haantjes module if and only if
\begin{equation}
\mathcal{H}_{\boldsymbol{A,B}} (X,Y)=\bs{0} \qquad\qquad \forall X,Y\in \mathfrak{X}(M) \ .
\end{equation}
\end{theorem}
\begin{proof}
As $\bs{A}$ and $\bs{B}$ are semisimple  commuting operators, they share a local eigenframe. In this eigenframe the two operators take simultaneously a diagonal form. Since they are also Haantjes operators, all of the brackets   $\mathcal{H}_1(\boldsymbol{A,B}) $, $\mathcal{H}_2(\boldsymbol{A,B})$ and $\mathcal{H}_2(\boldsymbol{B,A})$ identically vanish in view of Proposition \ref{pr:H1230}.  Thus, from Corollary  \ref{Ch:21} the result follows.
\end{proof}

\subsection{Spectral Analysis}
At this stage, we wish to discuss some of the spectral properties of non-semisimple  operators  on a manifold from the perspective of the theory of higher level Nijenhuis torsions. Let us denote  by $Spec(\boldsymbol{A}):=\{ \lambda_1(\boldsymbol{x}),
 \lambda_2(\boldsymbol{x}), \ldots, \lambda_s(\boldsymbol{x})\}$  the set of the  distinct eigenvalues of an operator $\boldsymbol{A}: \mathfrak{X}(M)\to \mathfrak{X}(M)$. In the forthcoming considerations, we shall always assume that these eigenvalues are \emph{real} and pointwise distinct. We denote by
 \begin{equation} \label{eq:DisL}
 \mathcal{D}_i(\boldsymbol{x}) = \ker \Big(\boldsymbol{A}(\boldsymbol{x})-\lambda_i(\boldsymbol{x})\boldsymbol{I}\Big)^{\rho_i}, \qquad i=1,\ldots,s
 \end{equation}
the \textit{i}-th generalized eigen-distribution of  index $\rho_i$, that is  the distribution of all the  generalized eigenvectors  corresponding to the eigenvalue $\lambda_i$. In Eq. \eqref{eq:DisL}, $\rho_i$ stands for the Riesz index of $\lambda_i$, which is the minimum integer such that
\begin{equation} \label{eq:Riesz}
\ker \Big(\boldsymbol{A}(\boldsymbol{x})-\lambda_i(\boldsymbol{x})\boldsymbol{I}\Big)^{\rho_i}\equiv \ker \Big(\boldsymbol{A}(\boldsymbol{x})-\lambda_i(\boldsymbol{x})\boldsymbol{I}\Big)^{\rho_{i}+1} \ ;
\end{equation}
 we also assume that $\rho_i$ is (locally) independent of $\boldsymbol{x}$.
When $\rho_i=1$,  $\mathcal{D}_i$ is a proper eigen-distribution.
Hereafter, unless differently stated, we shall use the adjective ``generalized'' to include the case of proper eigen-distributions as well.

In several applications, it is also useful to consider the action of our generalized torsions of any level on the generalized eigenvectors of $\boldsymbol{A}$.
Inspired by a formula for the Nijenhuis torsion evaluated on eigenvectors (proved in Appendix 6.2), we construct a generalized expansion, in terms of commutators, for the torsions of any level. It can be proved by induction over the integers $m\geq 2$ via a direct procedure.
\begin{proposition} \label{lm:TmL2autog}
Let $\boldsymbol{A}$ be a (1,1) tensor and $X_\alpha$, $Y_\beta$ be two of its 
 generalized eigenvectors  of $\mathcal{D}_\mu$,  $\mathcal{D}_\nu$, respectively.
 Then, for any integer $m\geq 2$ the following formula holds:
\begin{equation}\label{eq:TmL2autog}
\mathcal{\tau}^{(m)}_ {\boldsymbol{A}} (X_\alpha, Y_\beta)=
\sum_{i,j=0}^{m}(-1)^{i+j}\binom{m}{i}\binom{m}{j} \Big(\boldsymbol{A}-\mu\mathbf{I}\Big)^{m-i}\Big(\boldsymbol{A}-\nu \mathbf{I}\Big)^{m-j}
 [X_{\alpha-i}, Y_{\beta-j}].
\end{equation}
\end{proposition}
\noi This proposition will be useful in the proof of Lemmas \ref{lm:comXl} and \ref{lm:comXmn}, stated below.

\section{Generalized Nijenhuis torsions and Haantjes  brackets for nilcyclic operators}
In order to clarify the geometric relevance of both the generalized Nijenhuis torsions and the   Haantjes bracket of level $m$, we shall focus first on the case of nilcyclic operators (namely operators which are both \textit{nilpotent} and \textit{cyclic}). According to the classical Jordan-Chevalley decomposition theorem, given a vector space $V$, any linear endomorphism $\bs{L}: V\to V$ with real eigenvalues
can be written in a unique way as the sum  $\bs{L= D+ N}$, where $\bs{D}$ is a diagonalizable operator and $\bs{N}$ is a nilpotent operator, commuting with $\bs{D}$. 

\vspace{2mm}


Hereafter, the symbol $\langle \hspace{1mm}  \rangle $ will denote a $C^{\infty}(M)$-linear span of vector fields. 

\begin{definition}[Natural flag]
\par
Let $(U, x^1,\ldots, x^n)$  be a local coordinate chart and $\big(\frac{\partial}{\partial x^1}, \ldots,\frac{\partial}{\partial x^n}\big)$ the natural reference frame associated with it.  The flag  of integrable distributions
$$
\mathcal{C}_0=\{0\}\subset
\mathcal{C}_1=<e_1 >\subset
\mathcal{C}_2 =<e_1 , e_2>\subset \ldots\subset
\mathcal{C}_{n-1}=<e_1,\ldots,  e_{n-1}>
\subset \mathcal{C}_{n}=\mathfrak{X}(U),
$$
where $e_i:=\frac{\partial}{\partial x^i}$ $(i=1,\ldots,n-1)$,
will be called the natural flag associated with the local chart $(x^1,\ldots, x^n)$.
\end{definition}
\par
\subsection{Triangular form of nilcyclic operators} Let $M$ be an $n$-dimensional  differentiable manifold, and $\bs{A}:\mathfrak{X}(M)\to \mathfrak{X}(M)$ be a
 \textit{nilcyclic} \cite{BW} operator, that is   a nilpotent (1,1) tensor of maximal index $n$:
$$
\bs{A}^n=\bs{0} \qquad \textrm{and} \qquad \bs{A}^{n-1} \neq \bs{0} \ .
$$
This condition implies that there exist local reference frames, possibly non integrable ones, in which $\bs{A}$ is represented by a \textit{single}, upper strictly  triangular Jordan block. Under these assumptions,
the characteristic  null  flag of $\bs{A}$
$$
\{\bs{0}\}\subset \ker \bs{A}\subset \ker \bs{A}^2 \subset\ldots\subset \ker \bs{A}^n=\mathfrak{X}(M)
$$
is a complete flag \cite{Rob}, that is, $rank(\ker \bs{A}^j)=j\ , j=1,\ldots,n$. Also, the following inclusions hold:
\begin{equation} \label{eq:FlagSinv}
\bs{A}^k( \ker \bs{A}^{j})\subseteq \ker \bs{A}^{j-k}, \qquad\qquad j \geq k=1,\ldots,n  \ .
\end{equation}
Let us assume that there exists a local coordinate chart $(x^1,\ldots,x^n) $ on $M$ where $\bs{A}$ takes the upper strictly triangular form
\begin{equation} \label{Anoper}
\boldsymbol{A}=\sum _{i,j=1}^n a^{i}_{j}(\boldsymbol{x}) \frac{\partial}{\partial x^i}\otimes \rd x^j \ , \quad a^{i}_{j}=0 \hspace{2mm} \text{if} \hspace{2mm} i\geq j \ .
 \end{equation}

\noi Here $a^{i}_{j}(\bs{x})=a^{i}_{j}(x^1,\ldots,x^n)$ are smooth arbitrary functions depending on the local coordinates on $M$.
In this case, the integrable distributions of the natural flag coincide with the kernels of the powers of the operator $\bs{A}$. Precisely,
\beq \label{inclusion}
\mathcal{C}_{j} =\ker \boldsymbol{A}_{|U}^{j} \ , \qquad \qquad j=1,\ldots ,n \ .
\eeq

The following result establishes a \textit{necessary} condition for a nilcyclic operator to be represented in the upper triangular form, in a suitable coordinate chart.
\begin{proposition} \label{theo:triangular}
Let $M$ be an $n$-dimensional differentiable manifold, and $\bs{A}: \mathfrak{X}(M) \to \mathfrak{X}(M)$ be a nilcyclic (1,1) tensor on $M$. If there exists a local chart where the operator $\bs{A}$ takes the triangular form \eqref{Anoper}, then  the generalized Nijenhuis torsion
of level $(k-1)$  vanishes for all $X,Y\in\ker \bs{A}^k$:
\begin{equation}\label{eq:TauGnull}
\mathcal{\tau}^{(k-1)}_{\boldsymbol{A}} (\ker \bs{A}^k,  \ker \bs{A}^k)=\bs{0} \qquad\qquad k\in \mathbb{N}\backslash \{0,1\} \ .
\end{equation}

\end{proposition}
\begin{proof}
\par
First, we observe that   the (strong) invariance conditions
\begin{equation} \label{eq:FlagInv}
\bs{A}^p(\mathcal{C}_j) \subseteq \mathcal{C}_{j-p} \qquad\qquad p=0,\ldots,n \ ,
\end{equation}
hold as a consequence of  relations \eqref{eq:FlagSinv} and \eqref{inclusion}. In the latter conditions, it is understood that $\mathcal{C}_{j-p} \equiv\mathcal{C}_0\,$ for $j\leq p$.

 Then, we can proceed by induction over $k=2,\ldots,n-1$. To this aim, notice that for $k=2$, we have
\[
\mathcal{\tau}_{\bs{A}}(e_1,e_2)=\bs{A}^2\cancel{ [e_1,e_2]}+[\cancel{\bs{A} e_1},\bs{A}e_2]-
\bs{A}([\cancel{\bs{A} e_1},e_2]+[e_1,\bs{A}e_2]) \\
=-\cancel{\bs{A} [e_1,\bs{A} e_2]} =\bs{0} \ .
\]
The first addend vanishes because both $e_1$,  $e_2$ are constant fields, the second and third one vanish since $e_1\in \ker\bs{A}$, whereas the last term is zero due to both the invariance condition \eqref{eq:FlagInv} and the obvious involutivity of
 $\ker\bs{A}$ (being  $\textit{rank}~(\ker\bs{A})=1$). Now we assume that
  \begin{equation}
\mathcal{\tau}^{(k-1)}_{\boldsymbol{A}} (e_i, e_j)=\bs{0} \ , \qquad\qquad i,j=1,\ldots, k \ .
\end{equation}
This hypothesis, jointly with Definition \ref{GNTn} and the $\bs{A}$-invariance of
$\ker \bs{A}^k$ implies
$$
\mathcal{\tau}^{(k)}_{\boldsymbol{A}} (e_i, e_j)=\bs{0} \ , \qquad\qquad i,j=1,\ldots,k \ .
$$
We are left with the terms
$$
\mathcal{\tau}^{(k)}_{\boldsymbol{A}} (e_i, e_{k+1}) \ , \qquad\qquad i=1,\ldots, k,
$$
which can be evaluated by means of Eq. \eqref{eq:K17}. We obtain
$$
\mathcal{\tau}^{(k)}_{\boldsymbol{A}}(e_i,e_{k+1})=
\sum_{p,q=0}^{k} (-1)^{-(p+q)} \binom{k}{p} \binom{k}{q} \bs{A}^{p+q}\Big [\bs{A}^{k-p} e_i, \bs{A}^{k-q} e_{k+1}\Big ] \qquad i=1,\ldots k \ .
$$
As $e_i\in \ker\bs{A}^k$, the addends corresponding to  $p=0$ vanish. Moreover, for $p>0$, by virtue of equation \eqref{eq:FlagInv},  the following inclusions hold:
\begin{multline}
\bs{A}^{p+q}\Big [\bs{A}^{k-p} e_i, \bs{A}^{k-q} e_{k+1}\Big ]
\subseteq
\bs{A}^{p+q} \Big [\mathcal{C}_{i-(k-p)},\mathcal {C}_{k+1-(k-q)}\Big]
\subseteq
\bs{A}^{p+q}(\mathcal {C}_{max(i-k+p, 1+q)} )\\
\subseteq \mathcal{C}_{-(p+q)+max(i-k+p,1+q)}=\mathcal {C}_0 \ .
\end{multline}
\end{proof}

\noi We can now infer a direct, but important consequence of Proposition \eqref{theo:triangular}.
\begin{corollary} \label{cor:ntriangular}
Let $M$ be an $n$-dimensional differentiable manifold, $n\geq 2$  and $\bs{A}: \mathfrak{X}(M) \to \mathfrak{X}(M)$ be a nilcyclic (1,1) tensor on $M$.
Then, the condition 
\beq
\mathcal{\tau}^{(n-1)}_{\boldsymbol{A}}(X,Y)=\bs{0}, \qquad  X,Y\in \mathfrak{X}(M) \ ,
\eeq
is necessary for the existence of a local chart where $\bs{A}$ takes the triangular form \eqref{Anoper}.
\end{corollary}
\begin{proof}
It is sufficient to apply Proposition \eqref{theo:triangular} to the torsion of level $k=n$ and to observe that $\ker \bs{A}^n= \mathfrak{X}(M)$, as $\bs{A}$ is nilcyclic.
\end{proof}

\vspace{2mm}

\noi Consider the slightly more general case of a tensor of the form 
\beq \label{eq:Ln}
\bs{L} = \lambda \bs{I}+ \bs{A}, \qquad\qquad \lambda \in C^{\infty}(M) \ ,
\eeq
where $\bs{A}$ is a nilcyclic operator.  We have the following result.
\begin{corollary} \label{theo:4}
Let $M$ be an $n$-dimensional differentiable manifold, $n \geq 3$ and $\bs{L}: \mathfrak{X}(M) \to \mathfrak{X}(M)$ be a (1,1) tensor of the form \eqref{eq:Ln}.
If there exists a local chart where $\bs{L}$ takes the triangular form
\begin{equation} \label{LnTriang}
\boldsymbol{L}= \sum _{i,j=1}^n \left(\lambda (\boldsymbol{x})\delta^{i}_{j}+ a^{i}_{j}(\boldsymbol{x}) \right) \frac{\partial}{\partial x^i}\otimes \rd x^j \ , \quad a^{i}_{j}=0 \hspace{2mm} \text{if} \hspace{2mm} i\geq j \ ,
 \end{equation}
then
\begin{eqnarray}
\mathcal{\tau}^{(k-1)}_{\boldsymbol{L}}(X \, ,Y)&=&\bs{0}  \ , \qquad \forall \hspace{1mm} X,Y \in \ker(\bs{L}-\lambda\bs{I})^k , \qquad 3\leq k \leq n-1 
\ ,
\end{eqnarray}
and
\begin{eqnarray}
\mathcal{\tau}^{(n-1)}_{\boldsymbol{L}}(X,Y)&=&\bs{0} \ , \qquad \forall \hspace{1mm} X,Y \in \mathfrak{X}(M)  \ . 
\end{eqnarray}
\end{corollary}
\begin{proof}
The previous relations hold as  a consequence of Proposition \ref{theo:triangular}, Corollary \ref{cor:ntriangular} and Corollary \ref{corollary:3}.
\end{proof}

\subsection{An open problem: A Jordan-Chevalley decomposition}
The relevance of the generalized Nijenhuis torsions in the study of nilcyclic operators suggests, in a natural way, an interesting problem: namely, to ascertain whether there exists a Jordan-Chevalley-type decomposition for generic operators. Precisely, we propose the following, general

\vspace{2mm}

\textbf{Problem}.
Let $M$ be an $n$-dimensional differentiable manifold. Determine under which conditions there exist coordinate charts on $M$ such that  an operator
$\bs{L}: \mathfrak{X}(M) \to \mathfrak{X}(M)$ can be decomposed into the sum of two operators $\bs{L= D+N}$, where $\bs{D}$ is a diagonal operator and $\bs{N}$ is an upper  strictly   triangular  operator, commuting with $\bs{D}$, of the form \eqref{Anoper}.
\subsection{Conjecture for higher Haantjes brackets}
Inspired by the previous discussion, we conjecture the following result (which has been tested in many examples).
\par
\begin{conjecture}
Let $M$ be an $n$-dimensional differentiable manifold, and $\bs{A}, \bs{B}: \mathfrak{X}(M) \to \mathfrak{X}(M)$ be two nilpotent commuting (1,1) tensors on $M$. The vanishing of their generalized  Haantjes bracket  of level $(n-1)$
\beq
\mathcal{H}^{(n-1)}_{\boldsymbol{A}, \bs{B}}(X,Y)=\bs{0}
\eeq
is a necessary condition for the existence of a local chart where the tensors $\bs{A}$, $\bs{B}$ take simultaneously the triangular form \eqref{Anoper}.
\end{conjecture}
\par

\section{Frobenius Integrability and a generalized Haantjes theorem}

\subsection{Integrability of eigen-distributions: Necessary and sufficient conditions}

The relevance of the new families of generalized torsions introduced in this paper is further illustrated by their strict relationship with the properties of integrability of the generalized eigen-distributions admitted by a $(1,1)$-tensor.

In the first part of the discussion, the eigenvalues and eigenvectors of operators are supposed to be known. However, this hypothesis will be removed in the statement of our Main Theorem: indeed, no knowledge a priori of the spectrum and the eigen-distributions of the operators involved will be assumed.
\begin{remark}
All the eigen-distributions considered are supposed to be \textit{regular}, that is they have constant rank on $M$. For involutive distributions, this condition is equivalent to their Frobenius integrability.
\end{remark}

\begin{definition}\label{def:mI}
Let us consider a set of distributions $\{\mathcal{D}_i, \mathcal{D}_j, \ldots, \mathcal{D}_k \}$.
We shall say that such distributions  are \textit{mutually integrable} if

(i) each of them is integrable;

(ii) any  sum $\mathcal{D}_i  + \mathcal{D}_j +\cdots +\mathcal{D}_k$ is also integrable.
\end{definition}
First we state the following
\begin{lemma} \label{lemma:5}
Let $\bs{A}:\mathfrak{X}(M)\to \mathfrak{X}(M)$ be a non-invertible  operator. For any $X,Y \in \ker\bs{A}$, we have
\beq \label{eq:taumtokerA}
\mathcal{\tau}^{(m)}_{\bs{A}}(X,Y)= \bs{A}^{2m} [X,Y], \qquad m\in \mathbb{N} \backslash \{0\} \ .
\eeq
\end{lemma}
\begin{proof}
Equation \eqref{eq:taumtokerA} comes from Eq. \eqref{eq:K17} taking into account  that
 the terms with  $p<m$ and $q<m$ vanish.
\end{proof}
Let us recall that the \textit{Riesz index of an operator} $\bs{A}$ is the Riesz index $\rho$ of its zero eigenvalue (supposed to be locally constant over $M$), namely the minimum integer $\rho$ that makes stationary the sequence
\begin{equation}\label{eq:kerFlag}
\{0\}\subset\ker\bs{A}\subset\ker\bs{A}^2\subset \ldots \subset\ker \boldsymbol{A}^\rho=\ker\boldsymbol{A}^{\rho+j}\subseteq \mathfrak{X}(M), \quad j\in \mathbb{N} \ .
\end{equation}
\begin{proposition} \label{pr:kerBrho}
Let $\bs{A}:\mathfrak{X}(M)\to \mathfrak{X}(M)$ be an operator   and $\rho$ its Riesz index.
The following conditions are equivalent:
\begin{itemize}
\item[1)]
the distribution $\ker \bs{A}^\rho$ is involutive;
\item[2)]
\begin{equation} \label{eq:tauInv}
\exists \,  m\in \mathbb{N}\backslash \{0\} \qquad\textit{such that}\quad
 \mathcal{\tau}_{\bs{A}^\rho}^{(m)}(\ker\bs{A}^\rho,\ker\bs{A}^\rho) =\bs{0} \ ;
\end{equation}
\item[3)]
\begin{equation} \label{eq:tauNull}
\forall \, m\in \mathbb{N}\backslash \{0\}, \qquad
 \mathcal{\tau}^{(m)}_{\bs{A}^\rho}(\ker\bs{A}^\rho,\ker\bs{A}^\rho) =\bs{0} \ .
\end{equation}
\end{itemize}
\end{proposition}
\begin{proof}
\par
$1)\Longleftrightarrow 2)$.
From Eq.  \eqref{eq:taumtokerA} applied to $\bs{A}^\rho$ we get
\begin{equation} \label{eq:taumtokerBrho}
\mathcal{\tau}^{(m)}_{\bs{A}^\rho}(\ker\bs{A}^\rho,\ker\bs{A}^\rho)
=
\bs{A}^{2\, \rho\, m} [\ker\bs{A}^\rho, \ker\bs{A}^\rho] \ .
\end{equation}
Consequently,
$$
[\ker\bs{A}^\rho, \ker\bs{A}^\rho] \subseteq \ker\bs{A}^{2 \rho m}\stackrel{\eqref{eq:kerFlag}}{=} \ker\bs{A}^\rho
$$
if and only if the l.h.s. of Eq. \eqref{eq:taumtokerBrho} vanishes for some  $m \in \mathbb{N}\backslash \{0\}$.
\par
$1) \Longrightarrow 3)$ It is a direct consequence of Eq. \eqref{eq:taumtokerBrho}. The converse statement can be proved by following the same reasoning used in the proof of the first equivalence.
\par
 \end{proof}
The equivalence of the conditions (2) and (3) can be geometrically interpreted by observing that, if the distribution $\mathcal{D}=\ker\bs{A}^\rho$ is integrable, then the operator $\bs{A}^\rho$ can be restricted to each integral leaf  of $\mathcal{D}$; besides, each of these restricted operators vanishes.

 \par
Thus, applying Proposition \ref{pr:kerBrho} to each operator $\boldsymbol{B}_{i}:=\boldsymbol{A}-\lambda_{i} \bs{I}$, we obtain a novel necessary and sufficient condition for the integrability of the generalized eigen-distributions of an operator with real eigenvalues.
\par
\begin{corollary} \label{theo:5}
Let $\bs{A}:\mathfrak{X}(M)\rightarrow \mathfrak{X}(M)$ be an operator  and $\mathcal{D}_i=\ker(\boldsymbol{A}-\lambda_{i} \bs{I})^{\rho_i}$, where $\lambda_i \in Spec(\bs{A})$. Then, the distribution $\mathcal{D}_i$ is involutive if and only if there exists $m\in \mathbb{N}\backslash \{0\}$, such that
\begin{equation}\label{eq:NulltauPLDi}
\mathcal{\tau}^{(m)}_{(\boldsymbol{A}-\lambda_{i} \bs{I})^{\rho_{i}}}(\mathcal{D}_i,\mathcal{D}_i)=\bs{0} \ .
\end{equation}
\end{corollary}

\begin{remark} \label{rm:taum}
The original Nijenhuis theorem \cite{Nij1951} was not stated in the general case of non-semisimple operators. However, the previous  analysis  allows us to conclude that both the Nijenhuis torsion and the higher-level ones are equally valid, from a theoretical point of view, to detect the integrability properties of the generalized eigen-distributions of a non-semisimple operator.

\end{remark}

\subsection{Main Theorem}
The results stated above provide new necessary and sufficient conditions for the integrability of eigen-distributions of generalized eigenvectors. However, as we have remarked, they require the knowledge \textit{a priori} of the eigenvalues and eigenvectors of the  considered operator.
Instead, in the spirit of the seminal theorems by Nijenhuis and Haantjes,  it is desirable to have integrability conditions which do not require to solve explicitly eigenvalue problems, since this task is computationally intractable for large values of $n$. To this aim, we shall propose a novel strategy, based on the notion of generalized Nijenhuis tensors.

\vspace{2mm}

Formally, the problem we shall address is the following: to establish the conditions ensuring a priori the integrability of the generalized eigen-distributions of an operator $\bs{A}$ whose Haantjes torsion does not vanish, without recurring to the explicit determination of its eigen-distributions. To the best of our knowledge, no result is known regarding this problem. In the main Theorem stated below, we  will offer a  solution to this problem proposing a family of \textit{sufficient} conditions for integrability. 

First, let us prove some preliminary results.

\begin{lemma}\label{lm:comXl}
Let $\boldsymbol{A}: \mathfrak{X}(M)\to \mathfrak{X}(M)$ be an operator, $\mu\in Spec(\boldsymbol{A})$     and $X_\alpha$, $Y_\beta\in \mathcal{D}_\mu$ two of its
 generalized eigenvectors, of index $\alpha$, $\beta$ respectively,  belonging to  (possibly different) Jordan chains.
 If there exists an integer $m \geq 1$ such that
 \begin{equation} \label{eq:HaaNullD}
\mathcal{\tau}^{(m)}_{\boldsymbol{A}}(\mathcal{D}_\mu,\mathcal{D}_\mu)= \bs{0},
\end{equation}
then we have:
\begin{equation}
[X_\alpha,Y_{\beta}]\in \ker\Big(\boldsymbol{A}-\mu \mathbf{I}\Big)^{\alpha+\beta+m}
=\ker\Big(\boldsymbol{A}-\mu \mathbf{I}\Big)^{\mathrm{min}(\alpha+\beta+m,\rho_\mu)}
\subseteq \ker\Big(\boldsymbol{A}-\mu \mathbf{I}\Big)^{\rho_\mu} \ ,
\end{equation}
where $\mathrm{min( \cdot\,, \cdot\,)}$ stands for the minimum of its arguments.
\end{lemma}
\begin{proof}
First, we prove the case $m\geq 2$. If $\alpha=\beta=1$ and $\mu=\nu$, Eq. \eqref{eq:TmL2autog} implies that $[X_1,Y_1]\in \ker\Big(\boldsymbol{A}-\mu \mathbf{I}\Big)^{2m}$.
By induction over $(\alpha+\beta)$, and applying the operator $\Big(\boldsymbol{A}-\mu \mathbf{I}\Big)^{\alpha+\beta-m}$ to both members of Eq. \eqref{eq:TmL2autog} it follows that $$[X_\alpha,Y_{\beta}]\in \ker\Big(\boldsymbol{A}-\mu \mathbf{I}\Big)^{\alpha+\beta+m}.$$
In order to prove the case $m=1$, we observe that if the Nijenhuis torsion $\mathcal{\tau}^{(1)}_{\boldsymbol{A}}$ vanishes over the vector fields of $\mathcal{D}_\mu$, then $\mathcal{\tau}^{(m)}_{\boldsymbol{A}}$ vanishes as well, for $m\geq 1$.
\end{proof}

\begin{proposition} \label{pr:FXX}
Let $\boldsymbol{A}:\mathfrak{X}(M)\to \mathfrak{X}(M)$ be an  operator. Each of its generalized eigen-distributions  $\mathcal{D}_\mu$ with Riesz index $\rho_\mu \geq 1$
is involutive if
\begin{equation} \label{eq:HaaNullDD}
\mathcal{\tau}^{(m)}_ {\boldsymbol{A}}(\mathcal{D}_\mu,\mathcal{D}_\mu)= \bs{0} \ ,
\end{equation}
for some integer $m \geq 1$.
In addition, in the semisimple case  ($\rho_\mu=1$), if $\mathcal{D}_\mu$ is involutive, then condition \eqref{eq:HaaNullDD} is fulfilled for each integer $m \geq 2$.
 \end{proposition}
 \begin{proof}
Assuming condition  \eqref{eq:HaaNullDD}, Lemma \ref{lm:comXl} immediately  implies that    $\mathcal{D}_\mu$ is an involutive distribution, since
\begin{equation}\label{eq:FD}
[\mathcal{D}_\mu, \mathcal{D}_\mu] \subseteq  \mathcal{D}_\mu \ .
\end{equation}
In the specific case $\rho_\mu=1$, every $\mu$-eigenvector of $\boldsymbol{A}$ is a proper eigenvector, and from Eq. \eqref{eq:TmL2autog} for $m\geq 2$ one infers  that
$$
\mathcal{\tau}^{(m)}_{\boldsymbol{A}}(\mathcal{D}_\mu,\mathcal{D}_\mu)= 0 \Longleftrightarrow
[X_1,Y_1]\in \ker \Big(\boldsymbol{A}-\mu\mathbf{I}\Big)^{2m}=
\ker \Big(\boldsymbol{A}-\mu\mathbf{I}\Big) =\mathcal{D}_\mu\ .
$$
We deduce that for $\rho_\mu=1$, condition \eqref{eq:HaaNullDD} is also necessary for the involutivity of $\mathcal{D}_\mu$.
\end{proof}

\begin{lemma}\label{lm:comXmn}
Let $\boldsymbol{A}: \mathfrak{X}(M) \to \mathfrak{X}(M)$ be  an operator and $\mathcal{D}_\mu$, $\mathcal{D}_\nu$ two eigen-distributions satisfying, for some  integer $m \geq 1$, the condition
\begin{equation} \label{eq:HaaNullDmuDnu}
\mathcal{\tau}^{(m)}_ {\boldsymbol{A}}(\mathcal{D}_\mu,\mathcal{D}_\nu)= \bs{0} \ .
\end{equation}
Then, the commutator of two generalized eigenvectors of $\boldsymbol{A}$, with respect to two different eigenvalues $\mu$, $\nu$, satisfies
the property
\begin{eqnarray}
[X_\alpha,Y_{\beta}]&\in& \ker\Big(\boldsymbol{A}-\mu \mathbf{I}\Big)^{\alpha+m-1}\oplus
\ker\Big(\boldsymbol{A}-\nu \mathbf{I}\Big)^{\beta+m-1} \\
\nonumber
&=& \ker\Big(\boldsymbol{A}-\mu \mathbf{I}\Big)^{\mathrm{min}(\alpha+m-1,\rho_\mu)}\oplus
\ker\Big(\boldsymbol{A}-\nu \mathbf{I}\Big)^{\mathrm{min}(\beta+m-1,\rho_\nu)} \\
\nonumber
&\subseteq&
 \ker\Big(\boldsymbol{A}-\mu \mathbf{I}\Big)^{\rho_\mu}\oplus
\ker\Big(\boldsymbol{A}-\nu \mathbf{I}\Big)^{\rho_\nu} \ ,
\end{eqnarray}
with $1\leq \alpha\leq\rho_{\mu}$, $1\leq \beta\leq \rho_{\nu}$.
\end{lemma}
\begin{proof}
If $\alpha=\beta=1$ and $\mu\neq \nu$, Eq. \eqref{eq:TmL2autog} for $m\geq 2$  implies that $$[X_1,Y_1]\in \ker \Big(\boldsymbol{A}-\mu \mathbf{I}\Big)^{m}\oplus \ker\Big(\boldsymbol{A}-\nu \mathbf{I}\Big)^{m}.$$
By induction over $(\alpha+\beta)$, the result follows for $m\geq 2$ applying the operator $\Big(\boldsymbol{A}-\mu \mathbf{I}\Big)^{\alpha-1} \Big(\boldsymbol{A}-\nu \mathbf{I}\Big)^{\beta-1}$ to both members of Eq. \eqref{eq:TmL2autog}. If $\mathcal{\tau}^{(1)}_ {\boldsymbol{A}}(\mathcal{D}_\mu,\mathcal{D}_\nu)= 0
$, we also have $\mathcal{\tau}^{(m)}_ {\boldsymbol{A}}(\mathcal{D}_\mu,\mathcal{D}_\nu)= 0$
 for $m\geq 1$. This completes the proof.
\end{proof}
\noi
The latter Lemma also implies
$[\mathcal{D}_\mu , \mathcal{D}_\nu] \subset  \mathcal{D}_\mu \oplus \mathcal{D}_\nu$. This observation ensures the validity of the next result.
\begin{proposition}\label{pr:FXmn}
Let $\boldsymbol{A}: \mathfrak{X}(M) \to \mathfrak{X}(M)$ be  an  operator  and $\mathcal{D}_\mu$, $\mathcal{D}_\nu$ two eigen-distributions with Riesz indices $\rho_\mu$, $\rho_\nu$ respectively. Assume that for some  $m \geq 1$,
\begin{equation} \label{eq:HaaNullDmuDnu}
\mathcal{\tau}^{(m)}_ {\boldsymbol{A}}(\mathcal{D}_\mu,\mathcal{D}_\nu)= \bs{0} \ .
\end{equation}
Then the distribution
$$\mathcal{D}_\mu  \oplus \mathcal{D}_\nu \equiv \ker \Big(\boldsymbol{A}-\mu\mathbf{I}\Big)^{\rho_\mu }\oplus   \ker \Big(\boldsymbol{A}-\nu\mathbf{I}\Big)^{\rho_\nu }\ ,\qquad \mu \neq \nu$$
is involutive.
 \end{proposition}

Now, we can prove our main result concerning the mutual integrability of the eigen-distributions of operators.
\begin{theorem} \label{th:MainR}
Let $\bs{A}: \mathfrak{X}(M)\to \mathfrak{X}(M)$ be an operator. Assume that
\beq\label{eq:integ}
\tau^{(m)}_{\bs{A}}(X,Y)=\bs{0}, \qquad X, Y \in \mathfrak{X}(M)
\eeq
for some $ m \geq 1$. Then, each generalized eigen-distribution of $\bs{A}$ as well as each direct sum of its eigen-distributions is integrable.
\end{theorem}
\begin{proof}
This result is a direct consequence of Lemma \ref{lm:comXl} and Proposition \ref{pr:FXX}, whose hypotheses are indeed fulfilled once we assume the validity of condition \eqref{eq:integ}.
\end{proof}

\subsection{Block-diagonalization}
As a nontrivial application of Theorem \eqref{th:MainR}, we shall prove that, given an operator $\bs{A}$,  condition 
\eqref{eq:integ} is also sufficient to ensure the existence of a local chart where the operator $\bs{A}$ can be \textit{block-diagonalized}. Potentially relevant applications can be found, for instance, in the theory of hydrodynamic-type systems \cite{BogJMP}, in the study of partial separability of Hamiltonian systems \cite{CR2019} and, more generally, in the context of Courant's problems for
first-order hyperbolic systems of partial differential equations \cite{CH1962}.
 
Let $\bs{A}$ be an operator satisfying condition \eqref{eq:integ}; we denote by $r_i$ the rank of the distribution $\mathcal{D}_i$ of $\bs{A}$. We also introduce the  distribution (of corank $r_i$) 
\begin{equation}
 \label{eq:E}
 \mathcal{E}_i := Im \Bigl(\boldsymbol{A}-\lambda_i\mathbf{I}\Bigr)^{\rho_i }=  \bigoplus_{{j=1,\, j\neq i}} ^s  \mathcal{D}_j, \qquad\qquad i=1,\ldots,s
\end{equation}
which is  spanned by all the generalized eigenvectors of  $\boldsymbol{A}$, except those associated with the eigenvalue $\lambda_i$ (we remind that $\boldsymbol{A}$  by default has real eigenvalues). We shall say that $ \mathcal{E}_i$ is a \emph{characteristic distribution} of  $\boldsymbol{A}$. Let $\mathcal{E}^{\circ}_{i}$ denote the annihilator of the distribution $\mathcal{E}_{i}$. 
The   cotangent spaces of $M$ can be decomposed as
\begin{equation}
 \label{eq:TMdscomp}
T_{\boldsymbol{x}}^*M=\bigoplus_{i=1} ^s  \mathcal{E}_{i}^{\circ}(\boldsymbol{x}).
\end{equation}
 As a consequence of Theorem \ref{th:MainR},   each  characteristic distribution    $\mathcal{E}_i$  is  integrable. We shall denote by $ \mathrm{E}_i$ the foliation associated  with $\mathcal{E}_i$  and by
$E_i(\boldsymbol{x})$ the connected leave through $\boldsymbol{x}$,  belonging to $ \mathrm{E}_i$. Given  the set  of distributions $\{\mathcal{E}_1, \mathcal{E}_2, \ldots ,\mathcal{E}_s\}$,
we have associated an equal number of foliations $\{ \mathrm{E}_1,  \mathrm{E}_2, \ldots ,  \mathrm{E}_s\}$.   This set of foliations
 is referred to as the \textit{characteristic  web} of $\boldsymbol{A}$ and the leaves $E_i(\boldsymbol{x})$ of each foliation $ \mathrm{E}_i$ as the  \emph{characteristic fibers} of the web.
  
  \begin{definition}
Let $\boldsymbol{A}: \mathfrak{X}(M)\to \mathfrak{X}(M)$ be an operator satisfying Eq. \eqref{eq:integ}. A collection of $r_i$ smooth functions will be said to be adapted to the foliation $\mathrm{E}_i$ of the  characteristic  \textit{web} of $\bs{A}$  if the level sets of such functions coincide with the characteristic fibers of $\mathrm{E}_i$.
\end{definition}
\begin{definition}
Let $\boldsymbol{A}: \mathfrak{X}(M)\to \mathfrak{X}(M)$ be an operator satisfying Eq. \eqref{eq:integ}. A parametrization of the characteristic web of $\boldsymbol{A}$   is an ordered set  of $n$ independent smooth functions listed as
$(\boldsymbol{f}^1, \ldots,\boldsymbol{f}^i,\ldots, \boldsymbol{f}^s) $,
such that for any $i=1,\ldots, s$, the ordered subset
$\boldsymbol{f}^i=(f^{i,1}, \ldots , f^{i,r_i})$ is adapted to the  $i$-th characteristic foliation of the web:
\begin{equation}
\label{eq:fad}
f^{i,k}_{\vert   E_i(\mathbf{x})}=c^{i,k}  \qquad \forall E_i (\mathbf{x})\in  \mathrm{E}_i \ ,\quad
k=1,\ldots,r_i\ .
\end{equation}
Here $c^{i,k}$ are real constants depending  on the indices $i$ and $k$ only.
In this case, we shall say that the  collection of these functions is adapted to the web and that each of them is a \emph{characteristic function}.
\end{definition}

\begin{proposition}\label{cor:Hframe}
Let $\bs{A}: \mathfrak{X}(M) \to \mathfrak{X}(M)$ be an operator. 
\noi If 
 \beq
 \tau^{(m)}_{\bs{A}}(X,Y)=\bs{0} \ ,\qquad  X, Y \in \mathfrak{X}(M)
 \eeq
 for some $m\geq 1$, then  $\boldsymbol{A}$ admits local charts where it takes a block-diagonal form. 

\end{proposition}
\begin{proof}
Theorem \ref{th:MainR} ensures that each characteristic distribution $\mathcal{E}_i$ is integrable. Thus, we can also deduce the existence of $r_i$ exact one-forms $(\rd x^{i,1}, \ldots, \rd x^{i,r_i})$ in the corresponding annihilator $\mathcal{E}_i^\circ$; consequently, there exist functions
$\boldsymbol{x}^{i}=(x^{i,1}, \ldots, x^{i,r_i})$ adapted to the characteristic foliation $\mathrm{E}_i$. Collecting together all these functions, we get $n$ coordinates  $(\boldsymbol{x}^1, \ldots,\boldsymbol{x}^i,\ldots, \boldsymbol{x}^s)$ and therefore, a local chart $\{U, (\boldsymbol{x}^1, \ldots, \boldsymbol{x}^i,\ldots, \boldsymbol{x}^s)\}$, adapted to the characteristic web.
The natural frame associated 
$\left\{\frac{\partial}{\partial \boldsymbol{x}^1}, \ldots, \frac{\partial}{\partial \boldsymbol{x}^i},
\ldots,\frac{\partial}{\partial \boldsymbol{x}^s}\right\}$
is  a generalized eigen-frame. To prove this, it is sufficient to observe that the following decomposition holds:
\begin{equation}
 \label{eq:Dann}
\mathcal{D}_i^\circ =\bigoplus_{{j=1,\, j\neq i}} ^s  \mathcal{E}_{j}^{\circ} \ .
\end{equation}
Thus, any  generalized eigenvector  $W \in \mathcal{D}_i$ leaves invariant all the coordinate functions except at most the characteristic functions $\boldsymbol{x}^i=(x^{i,1}, \ldots, x^{i,r_i})$ of $\mathrm{E}_i$.
Thus, we deduce that 
\[
W=W(\boldsymbol{x}^i)\frac{\partial}{\partial \boldsymbol{x}^i}= \sum_{k=1}^{r_i}  W(x^{i,k}) \frac{\partial}{\partial x^{i,k}} \ .
\]
Therefore
\begin{equation}
\label{eq:Dbase}
\mathcal{D}_{i_{\vert U}}=\left\langle \frac{\partial}{\partial x^{i,1}},\ldots,
\frac{\partial}{\partial x^{i,r_i}}\right\rangle \ .
\end{equation}
This means that each frame \emph{equivalent} to $\left\{\frac{\partial}{\partial \boldsymbol{x}^1}, \ldots, \frac{\partial}{\partial \boldsymbol{x}^i},\ldots,\frac{\partial}{\partial \boldsymbol{x}^s}\right\}$ is an \textit{integrable eigen-frame} of generalized eigenvectors. Consequently, there exists an equivalence class of integrable frames, with their   local charts associated.  In these charts, the operator $\boldsymbol{A}$,  due to the invariance of its eigen-distributions, takes a block-diagonal form.
\end{proof}

\subsection{A comparison with Haantjes's classical theorem}
In his seminal paper \cite{Haa1955}, Haantjes proved the following,  fundamental theorem:
\par
i) If $\bs{A}$ is a \textit{semisimple} operator, the vanishing of its Haantjes torsion
\beq\label{Haanzero}
\mathcal{H}_{\bs{A}}(X,Y)=\bs{0} \ , \qquad\qquad \forall~X, Y \in \mathfrak{X}(M)
\eeq
 is a necessary and sufficient condition for the integrability of all of its eigen-distributions and  direct sums of them.
\par
ii) If $\bs{A}$ is \textit{non-semisimple}, then condition \eqref{Haanzero} is \textit{sufficient} to guarantee the integrability of its generalized eigen-distributions, but it is not necessary.

Our improvement of the Haantjes theorem consists in the family of  conditions \eqref{eq:integ}, which indeed are more general than the standard vanishing condition of the Haantjes torsion. Indeed, given a  non-semisimple operator $\bs{A}$, no conclusion about integrability of its eigen-distributions can be deduced from the Haantjes theorem,  if $\mathcal{H}_{\bs{A}}(X,Y)\neq \bs{0}$. However,  if there exists $m > 2$ such that $\tau^{(m)}_{\bs{A}}(X,Y)=\bs{0}$, this weaker condition is sufficient to ensure integrability.

In the semisimple case, $\rho_i=1$ $\forall i=1,\ldots,s$, so we recover Haantjes's result on  integrability directly from Proposition \ref{pr:FXX}. Instead, in the most general, non-semisimple case ($\rho_i>1$),  Theorem \ref{th:MainR}  provides an infinite family of new \textit{sufficient} conditions.

The following,  simple example can illustrate the potential relevance of Theorem \ref{th:MainR} in applicative contexts. Indeed, already in the case $n=3$ a generic non-semisimple operator is not necessarily a Haantjes one. Therefore, the Haantjes theorem does not apply. However, in our example, the associated generalized tensor of level three vanishes; this ensures integrability.

\begin{example}
Let $M$ be a $3$ dimensional manifold and $(x^1,x^2,x^3)$ a local chart in $M$. Consider  the operator 

\bea
\boldsymbol{L}(\boldsymbol{x})&=&\lambda_{1}(\boldsymbol{x})\bigg( \frac{\partial}{\partial x^1}\otimes \rd x^1 +\frac{\partial}{\partial x^2}\otimes \rd x^2 \bigg)+ \lambda_{2}(\boldsymbol{x}) \frac{\partial}{\partial x^3}\otimes \rd x^3   \\ \nn 
&+& f(\boldsymbol{x})\frac{\partial}{\partial x^1}\otimes \rd x^2 +g(\boldsymbol{x}) \frac{\partial}{\partial x^2}\otimes \rd x^3  \ ,
\eea
with $\la_1, \la_2$, $f,g\in C^{\infty}(M)$,  $\la_1\neq \la_2$.  A direct calculation shows that, for generic choices of these functions, the Nijenhuis and Haantjes torsions  do not vanish identically; however,
$\tau^{(3)}_{\bs{L}}(X,Y)=\bs{0}$. Therefore, according to Theorem \ref{th:MainR}, the generalized eigen-distributions of $\bs{L}$ are mutually integrable. 
To construct them explicitly, observe that the   minimal polynomial of $\bs{L}$ is $m(\lambda):=(\lambda-\lambda_1)^2(\lambda-\lambda_2)$, so that the Riesz indices of $\lambda_1$ and $\lambda_2$ are $\rho_1=2$, $\rho_2=1$, respectively. We obtain the generalized  eigen-distribution
$\mathcal{D}_1=\ker(\bs{L}-\lambda_i \bs{I})^2=\langle \frac{\partial}{\partial x^1},\frac{\partial}{\partial{x^2}} \rangle$, which is trivially integrable,  as well as the proper eigen-distribution $\mathcal{D}_2=\ker(\bs{L}-\lambda_2 \bs{I})=\langle X_{\lambda_2} \rangle$, with 
\[
X_{\lambda_2} =f g\, \frac{\partial}{\partial x^1} +(\lambda_2-\lambda_1)g\, \frac{\partial}{\partial x^2} +(\lambda_1-\lambda_2)^2\, \frac{\partial}{\partial x^3}
\ .
\] 
The latter eigen-distribution  is of  rank $1$ and obviously integrable. Thus, as $\mathcal{D}_1=\mathcal{E}_2$ and $\mathcal{D}_2= \mathcal{E}_1$, we get the spectral decompositions of the tangent spaces $T_{\bs{x}} M=\mathcal{D}_1 \oplus \mathcal{E}_1=\mathcal{D}_2 \oplus \mathcal{E}_2$. Correspondingly, for the cotangent spaces, we obtain  $T^{*}_{\bs{x}}M=\mathcal{E}_1^{\circ} \oplus \mathcal{E}_2^{\circ}$, where the annihilators of the characteristic distributions of $\bs{L}$ are
$$
\mathcal{E}_1^{\circ}= \langle (\lambda_1-\lambda_2)\rd x^1+fg\,\rd x^3 , (\lambda_1-\lambda_2)\rd x^1+f\,\rd x^2\rangle \ , \qquad 
\mathcal{E}_2^{\circ}=\langle \rd x_3 \rangle \ . 
$$
In order to construct explicitly a local chart where $\bs{L}$ takes a block-diagonal form (as ensured by Proposition \ref{cor:Hframe}), let us consider the space $\mathbb{R}^3$ endowed with Cartesian coordinates $(x^1,x^2,x^3)$.  We make the simple choice  $ \lambda_1= x^1 + x^2+ x^3$, $\lambda_2=x^1 + x^2$, $f= x^3$, $g=x^1$  in $M=R^{3}\backslash\{x^3=0\}$ (to guarantee $\la_1\neq \la_2$). By integrating the annihilators of the characteristic distributions (as explained in the proof of Proposition  \ref{cor:Hframe}), we find the local coordinate chart 
\[
y^1=x^1+x^2 \ , \qquad y^2= \frac{x^1}{x^3} \ , \qquad y^3= x^3 \ .
\]
On this chart, the operator $\bs{L}$ takes the block-diagonal form
\begin{equation}
\bs{L}=\left[
\begin{array}{cc|c}
y^1+2 y^3 & -(y^3)^2 & 0 \\
1 & y^1 & 0 \\
\hline 
0 & 0 & y^1
\end{array}
\right] \ .
\end{equation}

As we have shown, in the case of non-semisimple operators, the criterion of the vanishing of the Haantjes torsion, being only sufficient, may fail to detect the mutual integrability of the eigen-distributions even for very basic examples. Nevertheless, Theorem \ref{th:MainR} provides us with a more general tensorial  test,  guaranteeing integrability without the need for an explicit analysis of the eigen-distributions involved. Once integrability is ascertained, one can enter this kind of analysis in order to block-diagonalize the considered operator.

\end{example}

\begin{section}{Appendix}
\subsection{On the Haantjes bracket}
We  propose an explicit formula for the  Haantjes bracket of level $2$ in terms of commutators of vector fields:
\begin{eqnarray*} \label{bHtComm}
\nn &&\mathcal{H}_{\boldsymbol{A,B}} (X,Y):
=\frac{1}{2} \big(\bs{AB} +\bs{BA}\big)^2[X,Y]
+\boldsymbol{A B} \Big (2[\bs{A}X,\bs{B}Y]+2 [\bs{B}X, \bs{A} Y]+[\bs{AB}X,Y] \\  \nn && +[X,\bs{AB}Y]\Big )
 -2\bs{A}\big(\bs{AB}+\bs{BA} \big)\Big([\bs{B}X,Y]
+[X,\bs{B}Y] \Big)+[\bs{A}^2X,\bs{B}^2Y]+[\bs{AB}X,\bs{BA}Y] \\
  \nn &&- 2 \bs{A}\Big([\bs{BA}X,\bs{B}Y]+[\bs{A}X,\bs{B}^2Y]+  [\bs{AB}X,\bs{B}Y]+[\bs{B}^2X,\bs{A}Y]+[\bs{B}X,\bs{AB}Y] \\
  &&+ [\bs{BX},\bs{BA}Y]\Big )
+  \bs{A}^2\Big( [X,\bs{B}^2Y]+2[\bs{B}X,\bs{B}Y]+[\bs{B}^2X,Y] \Big)
\\ &&+ \text{symmetric terms in}~\bs{A}, \bs{B} \ .
 \end{eqnarray*}
\par
\subsection{The Nijenhuis torsion evaluated over eigenvectors} Let  $\boldsymbol{A}$ be an operator. Without loss of generality,  we shall focus only on two eigenvalues of $\boldsymbol{A}$,  $\mu=\mu(\boldsymbol{x})$ and $\nu=\nu(\boldsymbol{x})\in Spec(\boldsymbol{A})$, possibly coincident. Let us denote  by $X_\alpha$, $Y_\beta$  two   
 generalized eigenvectors, with indices $\alpha$ and $\beta$, associated with $\mu$ and $\nu$:
 \begin{equation}
X_\alpha \in \ker \Big(\boldsymbol{A}-\mu \mathbf{I}\Big)^{\alpha} \setminus \ker \Big(\boldsymbol{A}-\mu \mathbf{I}\Big)^{(\alpha-1)} ,\qquad
Y_\beta\in \ker \Big(\boldsymbol{A}-\nu \mathbf{I}\Big)^{\beta} \setminus \ker \Big(\boldsymbol{A}-\nu \mathbf{I}\Big)^{(\beta-1)} \ .
\end{equation}
They   belong to certain  Jordan chains defined in $\mathcal{D}_\mu$,  $\mathcal{D}_\nu$, respectively:
\begin{equation} \label{eq:Lautg}
\boldsymbol{A} X_\alpha = \mu X_\alpha +X_{\alpha -1},\qquad\boldsymbol{A} Y_\beta = \nu Y_\beta+Y_{\beta-1} \ ,\qquad 1\leq\alpha\leq \rho_\mu,\quad 1\leq\beta\leq\rho_\nu \ ,
\end{equation}
where $X_0$ and $Y_0$  are, by definition,  null vector fields.
Evaluating the Nijenhuis torsion on such eigenvectors, we obtain

\begin{eqnarray}\label{eq:TLautog}
\nn \mathcal{T}_ {\boldsymbol{A}} (X_\alpha, Y_\beta)&= &\Big(\boldsymbol{A}-\mu\mathbf{I}\Big) \Big(\boldsymbol{A}-\nu \mathbf{I}\Big) [X_{\alpha},Y_\beta]+
(\mu-\nu) \Big(X_\alpha(\nu) Y_\beta+Y_\beta(\mu)X_\alpha\Big)\\
 &-& \Big(\boldsymbol{A}-\mu\mathbf{I}\Big)[X_\alpha,Y_{\beta-1}]-\Big(\boldsymbol{A}-\nu\mathbf{I}\Big)[X_{\alpha-1},Y_\beta]+[X_{\alpha-1}, Y_{\beta -1}]\\
\nonumber&-&
\Big(X_\alpha(\nu) Y_{\beta-1}+Y_{\beta-1}(\mu)X_\alpha\Big)+
\Big(X_{\alpha-1}(\nu) Y_{\beta}+Y_{\beta}(\mu)X_{\alpha -1}\Big) .
\end{eqnarray}

\subsection{Haantjes brackets evaluated over common eigenvectors} Let $\bs{A}$ and $\bs{B}$ be two arbitrary (not necessarily Haantjes) operators, and let $X_{\mu }$ and $Y_{\nu }$ be two common eigenvectors. Precisely, let us consider  
\begin{equation}\label{eq:ABauto}
X_{\mu }\in \ker (\bs{A}-\mu_1\bs{I}) \cap \ker (\bs{B}-\mu_2\bs{I})\ , \quad Y_\nu \in \ker (\bs{A}-\nu_1 \bs{I}) \cap \ker (\bs{B}-\nu_2\bs{I})
\ .
\end{equation}
The Frolicher-Nijenhuis bracket satisfies the identity
\begin{equation}
\begin{split}
&\llbracket\boldsymbol{A,B}\rrbracket(X_{\mu},Y_{\nu}) =\bigg((\bs{A}-\mu_1\bs{I}) (\bs{B}-\nu_2\bs{I})+  (\bs{B}-\mu_2\bs{I})(\bs{A}-\nu_1\bs{I})\bigg) [X_{\mu} ,Y_{\nu}] \\
&+
\bigg( (\mu_1-\nu_1\bs) Y(\mu_2)   +  (\mu_2-\nu_2 )Y(\mu_1)\bigg)X_\mu +
\bigg( ( \mu_1-\nu_1) X_\mu(\nu_2)+(\mu_2-\nu_2)  X_{\nu}(\mu_1)\bigg) Y_\nu \ .
\end{split}
\end{equation}
Thus, we get
\begin{equation}
\mathcal{H}_{\boldsymbol{A,B}}(X_{\mu},Y)=\Big( (\bs{A}-\mu_1\bs{I}) (\bs{B}-\nu_2\bs{I})+ (\bs{B}-\mu_2\bs{I})  (\bs{A}-\nu_1\bs{I}) \Big)^2[X_{\mu} ,Y_{\nu}] 
\end{equation}
\begin{equation} \label{eq:H12XY}
\begin{split}
&\mathcal{H}_1(\boldsymbol{A,B})(X_{\mu},Y_{\nu})=(\bs{B}-\mu_2\bs{I})(\bs{B}-\nu_2\bs{I}) (\bs{A}-\mu_1\bs{I})(\bs{A}-\nu_1\bs{I}) [X_{\mu} ,Y_{\nu}] \\
&\qquad\qquad\qquad\qquad+ (\bs{A}-\mu_1\bs{I})(\bs{A}-\nu_1\bs{I}) (\bs{B}-\mu_2\bs{I})(\bs{B}-\nu_2\bs{I}) [X_{\mu} ,Y_{\nu}] \\
&\mathcal{H}_2(\boldsymbol{A,B})(X_{\mu},Y_{\nu})=\\
& \Big( (\bs{A}-\mu_1\bs{I}) (\bs{B}-\nu_2\bs{I})+ (\bs{B}-\mu_2\bs{I}) (\bs{A}-\nu_1\bs{I})\Big) 
(\bs{A}-\mu_1\bs{I})(\bs{A}-\nu_1\bs{I}) [X_{\mu} ,Y_{\nu}] \\
&+(\bs{A}-\mu_1\bs{I})(\bs{A}-\nu_1\bs{I})\Big( (\bs{A}-\mu_1\bs{I}) (\bs{B}-\nu_2\bs{I})+ (\bs{B}-\mu_2\bs{I}) (\bs{A}-\nu_1\bs{I})\Big) [X_{\mu} ,Y_{\nu}]\\
\end{split}
\end{equation}
\end{section}

\section*{Acknowledgement}

The authors gratefully thank  Prof. Y. Kosmann-Schwarzbach for useful discussions. Also, P. T. wishes to thank heartily Prof. N. Kamran for a careful reading of the manuscript, discussions and encouragement. We also wish to thank the Referees for many helpful suggestions.

This work has been partly supported by the research project PGC2018-094898-B-I00, MINECO, Spain, and by the ICMAT Severo Ochoa project SEV-2015-0554 (MINECO).
P. T. is member of the Gruppo Nazionale di Fisica Matematica (GNFM).

\end{document}